\documentclass[11pt, twoside]{article}
\usepackage{amsmath,amsthm,amssymb}
\usepackage{times}
\usepackage{enumerate}
\usepackage{latexsym}
\usepackage{amssymb}
\usepackage{amsmath}
\usepackage{fancybox}
\usepackage{epstopdf}
\usepackage{wasysym}
\usepackage{epsfig}
\usepackage{float}
\usepackage{amsfonts}
\usepackage{srcltx}
\usepackage{esint}
\usepackage[dvipsnames]{xcolor}
\usepackage{enumerate}

\numberwithin{equation}{section}

\def\titlerunning#1{\gdef\titrun{#1}}

\makeatletter

\def\author#1{\gdef\autrun{\def\and{\unskip, }#1}\gdef\@author{#1}}

\def\address#1{{\def\and{\\\hspace*{18pt}}\renewcommand{\thefootnote}{}%
\footnote {#1}}%
\markboth{\autrun}{\titrun}}

\makeatother

\def\email#1{e-mail: #1}

\def\keywords#1{\par\medskip
\noindent\textbf{Keywords.} #1}

\def\subjclass#1{\par\medskip
\noindent\textbf{Mathematics Subject Classification (2010).} #1}

\def\ou{\overline{u}}

\def\into{\int_{\Omega}}

\def\rn{\mathbb{R}^{N}}
  
 \def\w{{\bf w}} \def\z{{\bf z}} 
 \def\vare{\varepsilon}

\DeclareMathOperator{\dive}{div}

\newcommand\boeta{\eta}

\newcommand\bxi{\mbox{\boldmath{$\xi$}}}

\renewcommand{\epsilon}{{\varepsilon}}

\newcommand{\M}{\mathcal{M}}

\newcommand{\R}{{\mathbb R}}

\renewcommand{\d}{\,{\mathrm d}}
 \newcommand{\eps}{{\varepsilon}}
 \def\1{\raisebox{2pt}{\rm{$\chi$}}}
\def\dys{\displaystyle}
\newcommand{\matdot}{}

\newcommand{\ignore}[1]{}

\def\z{{\bf z}}

\def\ou{\overline{u}}

\def\vare{\varepsilon}

 \def\dys{\displaystyle}
 \def\w{{\bf w}} \def\z{{\bf z}} 
\def\rn{\mathbb{R}^{N}}
\def\re{\mathbb{R}}

\newcommand{\haus}{\displaystyle {\mathcal H}^{N-1}}
 \newcommand{\res}               {\!\!\mathop{\hbox{
                                \vrule height 7pt width .5pt depth 0pt
                                \vrule height .5pt width 6pt depth 0pt}}
                                \nolimits}

\newcommand{\tg}{\tilde g}
\newcommand{\ue}{u_\eps}

\theoremstyle{definition}

\newtheorem{theorem}{Theorem}[section]

\newtheorem{lemma}[theorem]{Lemma}
\newtheorem{proposition}[theorem]{Proposition}
\newtheorem{definition}[theorem]{Definition}
\newtheorem{remark}[theorem]{Remark}
\newtheorem{example}[theorem]{Example}

\begin{document}

\titlerunning{Nonlinear diffusion in transparent media:  the resolvent equation}

\title{Nonlinear diffusion in transparent media:\\   the resolvent equation}

\author{Lorenzo Giacomelli \and Salvador Moll \and Francesco Petitta}

\date{}

\maketitle

\address{L. Giacomelli: SBAI Department,
 Sapienza University of Rome, Via Scarpa 16, 00161 Roma, Italy; \email{lorenzo.giacomelli@sbai.uniroma1.it}
\and
S. Moll: Departament d'An\`{a}lisi Matem\`atica,
Universitat de Val\`encia,  Spain; \email{j.salvador.moll@uv.es}
\and
F. Petitta: SBAI Department,
 Sapienza University of Rome, Via Scarpa 16, 00161 Roma, Italy; \email{francesco.petitta@sbai.uniroma1.it}
}

\begin{abstract}
We consider the partial differential equation
$$
u-f=\dive\left(u^m\frac{\nabla u}{|\nabla u|}\right)
$$
with $f$ nonnegative and bounded and $m\in\re$. We prove existence and uniqueness of solutions for both the Dirichlet problem (with bounded and nonnegative {boundary datum}) and the homogeneous Neumann problem. Solutions, which a priori belong to a space of truncated bounded variation functions, are shown to have zero jump part with respect to the $\haus$ Hausdorff measure. Results and proofs extend to more general nonlinearities.



\end{abstract}

\section{Introduction}\label{sec1}

Let $\Omega$ be a bounded open set of $\rn$ with Lipschitz continuous boundary, $N\geq 1$, and  $m\in \re$. We are interested in the partial differential equation

\begin{equation}\label{pde-a}
u-f=\dive\left(u^m\frac{\nabla u}{|\nabla u|}\right) \quad \mbox{in }\ \Omega
\end{equation}
{with} $0\leq f\in L^{\infty}(\Omega)$. Equation \eqref{pde-a} corresponds to the resolvent equation of the following evolution equation:
\begin{equation}\label{ppde}
   \frac{\partial u}{\partial t}=\dive \left(u^m\frac{\nabla u}{|\nabla u|}\right).
 \end{equation}
When $m=0$, \eqref{ppde} coincides with the nowadays well-known {\it total variation flow}: we refer to the monograph \cite{ACMBook} for a detailed study of the subject
and to \cite{SapiroBook} for its applications in image processing. The case $m=1$ (the so-called {\it {heat equation in} transparent media}) was considered in \cite{ACMM_jee07}, where  existence and uniqueness of entropy solutions to the Cauchy problem for both \eqref{pde-a} and \eqref{ppde} were obtained. In addition, it was shown in \cite{ACMM_jee07} that solutions to the {\it relativistic heat equation}
\begin{equation}\label{ree}
\frac{\partial u}{\partial t}=\varrho\dive \left(u\frac{\nabla u}{\sqrt{u^2+\varrho^2|\nabla u|^2}}\right),
\end{equation}
converge to solutions of \eqref{ppde} (with $m=1$) as $\varrho\to +\infty$.  For $m>1$, equation \eqref{ppde} is the formal limit of the {\it relativistic porous medium equation},
\begin{equation}\label{pmrhe}
\frac{\partial u}{\partial t}=\varrho\dive \left(\frac{{u^m}\nabla u}{\sqrt{u^2+\varrho^2|\nabla u|^2}}\right), \quad m>1\,,
\end{equation}
as the kinematic viscosity $\varrho$ tends to $+\infty$ (here the maximal speed of propagation has been normalized to $1$). To the best of our knowledge, Eq. \eqref{pmrhe} was introduced in \cite{Rosenau_prA} while studying heat diffusion in neutral gases (precisely with $m=3/2$). Existence and uniqueness of solutions for the Cauchy problem associated to {\eqref{pmrhe}} were obtained in \cite{ACM_arma05}. Some key-features of solutions, such as propagation of support, waiting time phenomena, speed of discontinuity fronts, and pattern formations, have been recently addressed by many authors \cite{ACM_jde08,CC_na13,Caselles_jde11,Giacomelli_siam15,GMP, Calvo_siam12,CCCSS_jems16, CCCSS_in16}.

\smallskip

Three points of interest motivate the study of \eqref{ppde} and its resolvent equation also for $m\notin\{0,1\}$.

\smallskip

(I) {\it{Shock formation, $m>1$}}. Besides pioneering contributions \cite{BdP_arma92,Blanc_cpde93} and numerical simulations \cite{ACMSV_siam12,CCM_plms13}, the mechanism and the dynamics of shock formation for solutions to \eqref{pmrhe} is not yet fully understood (see in particular \cite{GMP} for further insights).
Since \eqref{ppde} and \eqref{pmrhe} formally coincide where $|\nabla u|\gg 1$, in particular at a discontinuity front, \eqref{ppde} may be seen as a prototype equation for investigating such phenomena. More generally, in flux-saturated diffusion equations such as \eqref{pmrhe}, one expects to see strong interplays between hyperbolic and parabolic mechanisms: the scaling invariance of \eqref{ppde} with respect to $x$ should make these interplays more transparent and easier to study qualitatively.

\smallskip

(II) {\it{Large solutions, $m<0$}}. The analysis of qualitative phenomena, namely the initial propagation of support, also motivates the analysis of \eqref{ppde} in the case $m<0$. Indeed, assume that we are in the case $N=1$ and that a solution to \eqref{pmrhe} has a fixed support $[a,b]$ during a time interval $(0,T)$ (in particular, $u(t,\cdot)$ is continuous and equals $0$ across its boundary, see \cite{Caselles_jde11}). Suppose that $u|_{t=0}$ (hence $u(t)$) has unit total mass. Let $\varphi(t,\eta)$ be defined through
$$
\int_{a}^{\varphi(t,\eta)} u(t,x) dx = \eta, \qquad   \eta\in (0,1).
$$
Formally, the equation satisfied by $v(t,\eta):=1/u(t,\varphi(t,\eta))$ is
{\begin{equation}
  \label{1d-lagr} \left\{\begin{array}{lc} \displaystyle\frac{\partial v}{\partial t}=\varrho \left(\frac{v^{{1-m}} v_\eta}{\sqrt{v^4+\varrho^2 v_\eta^2}}\right)_\eta & {\rm in \ } (0,T)\times (0,1)  \\ v=+\infty & {\rm on \ }[0,T]\times\{0,1\}\,,\end{array}\right.
\end{equation}}
\noindent i.e., $v$ is a ``large solution'' to (\ref{1d-lagr}.a). In \cite{CCM_plms13}, this lagrangian approach was used in the case $m=1$ to show some additional regularity properties for \eqref{ree} {(see also \cite{CCCSS_jems16} for the use of this approach respect to Eq. \eqref{flpme} below)}. Letting $\rho\to \infty$, one is led to analyze the problem of large solutions for Equation \eqref{pde-a} with $m<0$.

\smallskip

(III) {\it{Well-posedness}}. The last point of interest in \eqref{ppde} is of a more theoretical nature: \eqref{ppde} stands as a model for autonomous evolution equations in divergence form which, though of second order, have the same scaling of a first order nonlinear conservation law. As mentioned in (I), this structure may lead to simpler qualitative studies. However, at the level of well-posedness, it poses quite a few additional difficulties with respect to {\eqref{pmrhe}} and other flux-saturated diffusion equations, such as the speed-limited porous medium equation,
\begin{equation} \label{flpme}
  u_t=\dive\left(\frac{{u}\nabla  u^{M-1}}{\sqrt{1+|\nabla u^{M-1}|^2}}\right), \quad M>{1}.
  \end{equation}
Indeed, while an existence and uniqueness theory is available for both \eqref{pmrhe} and \eqref{flpme}, it is not yet for \eqref{ppde}. As first step toward the elaboration of such theory, the aim of this paper is to give an appropriate notion of solutions to \eqref{pde-a} and to discuss their existence and uniqueness.

\smallskip

We mainly concentrate on the Dirichlet problem,
\begin{equation} \label{pde}
\left\{
\begin{array}{ll}\displaystyle
u-f=\dive\left(u^m\frac{\nabla u}{|\nabla u|}\right) & \mbox{in }\ \Omega \\
u=g & \mbox{on}\  \partial\Omega\,,
\end{array}
\right.
\end{equation}
where $g\in L^\infty(\partial\Omega)$ is nonnegative. In fact, consistently with (II), for $m<0$ we assume that $f$ and $g$ (hence, as we shall see, solutions) are bounded away from zero. On the other hand, for $m>0$ a positive boundary datum $g$ does not guarantee positivity of the solution (see e.g. Example \ref{L-explicit}$(v)$ for $f=0$) and, moreover, the case $g=0$ is interesting in view of  the  relation between \eqref{ppde} and \eqref{pmrhe} (see (I) and (II) above). Therefore, for $m>0$ we only assume nonnegativity of the data.

\smallskip

For all $m\in\R$, we introduce a notion of solutions for problem \eqref{pde} (see Definitions \ref{defi1} and \ref{defi1d}) and we prove existence of solutions (see Theorems \ref{exi} and \ref{exid}) as well as a contraction principle in $L^1(\Omega)$ (see Theorems \ref{comp} and \ref{UniqEllipticd}). We also show that solutions of \eqref{pde} have diffuse gradients, i.e., their jump set has zero $(N-1)$-dimensional Hausdorff measure (see Lemma \ref{cont} and \ref{4.8}), an insight which applies as well to the resolvent equations of \eqref{pmrhe} and \eqref{flpme} (cf. Remark \ref{rem-cont}).

\smallskip

According to our notion of solution, the Dirichlet boundary condition $u=g$ transforms into obstacle-type constraints which formally read as follows:
\begin{eqnarray}\label{bc-}
& u\le g,  \quad\mbox{with}\quad \frac{Du}{|Du|}\cdot \nu =
1 \ \mbox{ if } \ u< g & \quad\mbox{when $m<0$},
\\ \label{bc+}
& u\ge g,  \quad\mbox{with}\quad \frac{Du}{|Du|}\cdot \nu =
-1 \ \mbox{ if } \ u>g & \quad\mbox{when $m>0$,}
\end{eqnarray}
where $\nu$ denotes the outward unit normal to $\partial\Omega$ (see e.g. \cite{ACMBook} for the case $m=0$, in which $u=g$ turns into $\frac{Du}{|Du|}\cdot \nu\in$ sign$(g-u)$). Now, it is not surprising that in the $BV$-framework the boundary datum may not be attained.
If this is the case, \eqref{bc-}$_2$ and \eqref{bc+}$_2$ are natural compatibility conditions: seen together, they formally say that, while approaching $\partial\Omega$, either $u$ strictly decreases toward $g$ if $u>g$, or viceversa. The selection criterium given by the sign of $m$ can then be understood by a simple heuristic in one space dimension: assuming that $u$ is strictly monotone near $\partial\Omega$, \eqref{pde} reduces to
\begin{equation}\label{heur1}
mu^{m-1}|u'|=u-f \quad\mbox{for d$(x,\partial\Omega)\ll 1$}.
\end{equation}
If for instance $m>0$, then \eqref{heur1} implies that $u=g$ can be attained only if $g-f\ge 0$, and otherwise $u\ge f>g$. The case $m<0$ is symmetric. Examples are given in Lemma \ref{L-explicit}(i).

\smallskip

Motivated by (II), we also provide preliminary information on existence or nonexistence of large solutions, i.e., solutions to
\begin{equation*}
\left\{
\begin{array}{ll}\displaystyle
u-f=\dive\left(u^m\frac{\nabla u}{|\nabla u|}\right) & \mbox{in }\ \Omega \\
u=+\infty & \mbox{on}\  \partial\Omega\,,
\end{array}
\right.
\end{equation*}
where $f\in L^{\infty}(\Omega)$. We show in particular that, when $m<0$ and $\Omega$ is a ball, solutions are bounded independently of the boundary datum, a phenomenon which occurs also for $m=0$ (see \cite{MP_jam15}, and \cite{MP_jfa12} for the corresponding parabolic problem).
On the other hand, for $m>1$  solutions with $g=n\in \mathbb N$ cannot converge to any $L^1_{loc}$ function in $\Omega$, i.e. large solutions should not exist.

\smallskip

A similar (though simpler) approach leads to analogous results for the homogeneous Neumann problem (see Section \ref{neumann}):
\begin{equation}\label{neumannpde}\left\{\begin{array}{ll}\displaystyle
u-f=\dive\left(u^m\frac{\nabla u}{|\nabla u|}\right) &  \mbox{in }\ \Omega \\ {u^m\frac{\nabla u}{|\nabla u|}\cdot \nu}=0 & \mbox{on \ }\partial \Omega. \end{array}\right.
\end{equation}
Also, our analysis of both \eqref{pde} and \eqref{neumannpde} extends to more general forms of the nonlinearities (see Section \ref{neumann}).

\smallskip

The plan of the paper is the following: Section \ref{preli} contains definitions, notations, and known results (on divergence-measure fields and TBV-functions) used in the paper. Section \ref{approx} is devoted to the  construction of suitable approximating solutions. Section \ref{sing} discusses well-posedness and regularity of solutions to \eqref{pde} in the singular case, $m<0$.
In Section \ref{4}, analogous results are proved for problem \eqref{pde} in the degenerate case, $m>0$, with some technical complications since a priori bounds do not control $|Du|$ down to $u=0$. Due to that, a few new results on $TBV$-spaces are given in Section \ref{5.1}. Section \ref{S-qp} discusses qualitative features of solutions to \eqref{pde}, including global a priori $L^\infty(\Omega)$ bounds of solutions ($m<0$), a barrier for the case $0<m<1$, and nonexistence of uniform bounds in case $m>1$. Section \ref{neumann} deals with the case of homogeneous Neumann boundary conditions and to more general nonlinearities.

\section{Preliminaries}\label{preli}

\subsection{Notation}
We denote by $\mathcal H^{N-1}$ the $(N-1)$-dimensional Hausdorff measure, by $\mathcal L^N$ the $N$-dimensional Lebesgue measure, and by ${\mathcal M}(\Omega)$ the space of  finite Radon measures on $\Omega$ (see \cite[Def. 1.40]{AFPBook}). The subscript $_0$ denotes spaces of compactly supported functions. We recall that ${\mathcal M}(\Omega)$ is the dual space of $C_0(\Omega)$. We let $\mathcal D(\Omega):=C_0^\infty(\Omega)$, $\mathcal D'(\Omega)$ its dual {, and
\begin{eqnarray*}
BV^+(\Omega) &=& BV(\Omega)\cap L^1(\Omega;[0,+\infty)),
\\
DBV(\Omega) &=& \{u\in BV(\Omega): \ \mathcal H^{N-1}(S_u)=0\}.
\end{eqnarray*}
}
We use standard notation and properties of $BV$ functions, for which we refer to \cite{AFPBook}. For $a < b$, we define the truncating functions
\begin{equation*}
T_{a}^{b}(s) := \max(\min(b, s ), a),  \quad  T_{a} (s) :=  T_{-a}^{a}(s), \quad T_{a}^{\infty} (s):= \max( s , a),\quad s\in \R,
\end{equation*}
and the spaces
$$
\mathcal{T} := \{T_{a}^{b} : 0<a<b\}\,, \quad  \mathcal{T}^\infty  := \{T_{a}^{\infty} : 0<a\}\,.
$$
For $F\in W^{1,1}_{loc}((0,+\infty))$, let
\begin{equation}\label{def-phi-F}
\phi_F(s):= \int_1^s F'(\sigma) \sigma^m\d \sigma, \quad s>0.
\end{equation}
In particular,
\begin{equation}\label{def-phi-I}
\phi(s):=\phi_{\mathrm{Id}}(s)=\left\{\begin{array}{ll}
\frac{1}{m+1} s^{m+1} & \mbox{ if $m\ne -1$} \\ \log s  & \mbox{ if $m= -1$,}\end{array}\right. \quad \mbox{so that} \quad \phi'(s)=s^{m}\,.
\end{equation}

{
\subsection{TBV-functions}
}

Let
\begin{eqnarray*}
TBV^+(\Omega) &=& \{u\in L^1(\Omega;[0,+\infty)): \ F(u)\in BV(\Omega) \ \mbox{ $\forall \ a>0$, $F\in W^{1,\infty}_{a}$}\},
\end{eqnarray*}
where
\begin{equation}\label{def-WW}
W^{1,\infty}_{a}=W^{1,\infty}([0,+\infty);[a,+\infty)), \quad a>0\,.
\end{equation}
{We now outline some properties of $TBV^+(\Omega)$ which are analogous to those of $GBV(\Omega)$, the space of integrable functions such that $T_{a}(u)\in BV(\Omega)$ for any $a\geq 0$ (see \cite{AFPBook}). Further properties of the space $TBV^{+}$ will be proved later in Section \ref{5.1}. First of all, $TBV^+$ may be equivalently defined as}
\begin{equation*}
TBV^+(\Omega)=\{u\in L^1(\Omega;[0,+\infty)): \ T(u)\in BV(\Omega) \ \mbox{for all $T\in \mathcal T^\infty$}\}
\end{equation*}
(see \cite[Remark 4.27]{AFPBook}). Given $u\in L^1(\Omega)$, the upper and lower approximate limits of $u$ at a point $x\in\Omega$ are defined respectively as
\begin{eqnarray*}
  u^\vee(x)&:=& \inf\{t\in \R : \lim_{\rho\downarrow 0}\rho^{-N} |\{u>t\}\cap B_\rho(x)|=0\},\\
  u^\wedge(x)&:=& \sup\{t\in \R : \lim_{\rho\downarrow 0}\rho^{-N} |\{u<t\}\cap B_\rho(x)|=0\}.
\end{eqnarray*}
We let $S_u^*:=\{x\in\Omega : u^\wedge(x)<u^\vee(x)\}$ and
$$
DTBV^+(\Omega)=\{u\in TBV^+(\Omega): \ \mathcal H^{N-1}(S_u^*)=0\}.
$$
The set of weak approximate jump points is the subset $J_u^*$ of $S_u^*$ such that there exists a unit vector $\nu_u^*(x)\in\R^N$ such that the weak approximate limit of the restriction of $u$ to the hyperplane $H^+:=\{y\in \Omega: \langle y-x,\nu_u^*(x)\rangle>0\}$ is $u^\vee(x)$ and the weak approximate limit of the restriction of $u$ to  $H^-:=\{y\in \Omega: \langle y-x,\nu_u^*(x)\rangle<0\}$ is $u^\wedge(x)$. In \cite[Page 237]{AFPBook} it is shown that for any $u\in L^1_{loc}(\Omega)$,  $J_u\subset J_u^*$. Moreover, $u^\vee(x)=\max\{u^+(x),u^-(x)\}$, $u^{\wedge}(x)=\min\{u^+(x),u^-(x)\}$ and $\nu_u^*(x)=\pm \nu_u(x)$ for any $x\in J_u$.  Furthermore, arguing as in \cite[Theorem 4.34]{AFPBook} {one obtains} the following result.

 \begin{lemma}\label{tbvjump}For any $u\in TBV^+(\Omega)\cap L^\infty(\Omega)$,  \begin{itemize}\item[(i)] $S_u^*=\cup_{a>0}S_{T_{a}^\infty(u)}$   and
 {
 \begin{equation*}
 u^\vee(x)=\lim_{a\to 0^+}(T_{a}^\infty(u))^\vee(x)\,,\quad u^\wedge(x)=\lim_{a\to 0^+}(T_{a}^\infty(u))^\wedge(x);
 \end{equation*}
  }
 \item[(ii)] $S_u^*$ is countably $\mathcal H^{N-1}$ rectifiable and $\mathcal H^{N-1}(S_u^*\setminus J_u^*)=0$.
\end{itemize}
\end{lemma}

\subsection{Divergence-measure vector-fields}\label{GreenAnz}

Let
$$
 X(\Omega) = \left\{ \z \in L^{\infty}(\Omega; \R^N) \ :
 \, \dive \z \in L^\infty(\Omega) \right\}\,,
$$
$$
 X_{\mathcal{M}}(\Omega) = \left\{ \z \in L^{\infty}(\Omega; \R^N) \ :
 \, \dive \z \in {\mathcal M}(\Omega) \right\}.
$$
In \cite[Theorem 1.2]{Anz_ampa83} (see also \cite{ACMBook,ChF_arma99}), the weak trace on $\partial \Omega$ of the normal component of $\z \in X_{\mathcal{M}}(\Omega)$ is defined as a linear operator $[\cdot,\nu]: X_{\mathcal{M}}(\Omega) \rightarrow L^\infty(\partial \Omega)$ such that
$
\Vert \, [\z,\nu] \, \Vert_{L^\infty(\partial \Omega)} \leq \Vert \z \Vert_{\infty}
$
for all $\z \in X_{\mathcal{M}}(\Omega)$ and $[\z,\nu]$ coincides with the point-wise trace of the normal component if $\z$ is smooth:
$$
[\z,\nu](x) = \z(x) \cdot \nu(x) \quad \hbox{for all} \ x \in \partial \Omega  \ \ \hbox{if} \ \z \in C^1(\overline{\Omega}, \R^m).
$$
It follows from \cite[Proposition 3.1]{ChF_arma99} or \cite[Proposition 3.4]{ACM_afst05} that
{$\dive \z$ is absolutely continuous with respect to $\mathcal H^{N-1}$.
}
Therefore, given  $\z \in X_{\mathcal{M}}(\Omega)$ and $u \in BV(\Omega)\cap L^\infty(\Omega)$, the functional $(\z,Du)\in \mathcal D'(\Omega)$ given by
\begin{equation}\label{defmeasx144}
\langle (\z,Du),\varphi\rangle := - \int_{\Omega} u^* \, \varphi \d (\dive \z) - \int_{\Omega}
u \, \z \matdot \nabla \varphi \d x\,
\end{equation}
is well defined, and the following  holds (see  \cite{Caselles_jde11},  Lemma 5.1, Theorem 5.3,  Lemma 5.4,  and Lemma 5.6).

\begin{lemma}\label{lemmacaselles}
Let $\z \in X_{\mathcal{M}}(\Omega)$ and  $u \in BV(\Omega)\cap L^\infty(\Omega)$. Then
the functional $(\z,Du)\in \mathcal D'(\Omega)$ defined by \eqref{defmeasx144} is a Radon measure which is absolutely continuous with respect to $\vert Du \vert$. Furthermore
\begin{equation}\label{Green}
\int_{\Omega} u^*\d(\dive \z) + (\z, Du)(\Omega) =
\int_{\partial \Omega} [\z, \nu] u \d\mathcal{H}^{m-1},
\end{equation}
\begin{equation}\label{anzellotti-caselles}
  \dive(u\z )=u^*\dive\z+(\z,Du)\,\quad {\rm as \ measures,}
\end{equation}
and
\begin{equation}\label{cas-trace}
[u\z ,\nu] = u[\z,\nu]   \quad\mbox{$\mathcal H^{N-1}$-a.e. on $\partial\Omega$}.
\end{equation}
\end{lemma}

We denote by $\theta(\z,Du)$  the Radon-Nikodym derivative of $(\z,Du)$ with respect to $|Du|$. The following result can be found in \cite[Proposition 2.7]{LS}.

\begin{lemma}
  \label{anz} Let $u\in DBV(\Omega)\cap L^\infty(\Omega)$, $\z\in X_\M(\Omega)$ and
       let $\Gamma$ be a Lipschitz continuous nondecreasing function. Then
   \begin{equation}\label{composition}
    \theta(\z,D(\Gamma(u)))=\theta(\z,D u) \quad |D(\Gamma\circ u)|{\rm -a.e. \ in \ } \Omega.
    \end{equation}
    Consequently,
    \begin{equation}\label{composition2}
    (\z,D(\Gamma(u)))=\Gamma'(u)(\z,Du) \quad\mbox{as measures.}
    \end{equation}
  \end{lemma}

In \cite[\S 3]{ACM_afst05} (see also \cite{Caselles_jde11}), the normal traces $[\z,\Sigma]^\pm$ of a vector field $\z\in X_{\mathcal{M}}(\Omega)$ are defined on an oriented $C^1$-hypersurface  $\Sigma \subset \Omega$:
$$
[\z,\Sigma]^\pm := [\z,\nu_{\Omega^\pm}],
$$
where $\Omega^\pm \Subset \Omega$ are open $C^1$ domains such that $\Sigma \subset \partial \Omega^\pm$ and $\nu_{\Omega^\pm} = \pm \nu_{\Sigma}$ (the definition is seen to be independent of $\Omega^\pm$ up to a set of zero $\haus$-measure).
In addition \cite[Proposition 3.4]{ACM_afst05}, it is proved that
\begin{equation}\label{Tresd}
(\dive \z) \res\Sigma = \left([\z,\Sigma]^+ - [\z,\Sigma]^-\right)\haus \res \Sigma.
\end{equation}

By localization, this notion is then extended to oriented countably $\haus$-rectifiable sets $\Sigma$ (these are countable union, up to a $\haus$-negligible set, of oriented $C^1$-hypersurfaces). Using this definition, from (\ref{Tresd}) one immediately gets the following:
\begin{lemma}\label{Ammbrdd} Let $\z \in X_{\mathcal{M}}(\Omega)$ and let $\Sigma \subset \Omega$ be an oriented countably $\haus$-rectifiable set. Then
$$
(\dive \z)\res\Sigma = \left([\z,\Sigma]^+ - [\z,\Sigma]^-\right)\haus\res\Sigma.
$$
\end{lemma}

The next result is a consequence of Lemma \ref{lemmacaselles}.
\begin{lemma}\label{lemamb}
 Let $u\in BV(\Omega)\cap L^\infty(\Omega)$ and $\w\in X_{\mathcal{M}}(\Omega,\R^N)$. Then  \begin{equation}
    \label{traceproduct} [u\w,\nu_u]^{\pm}= u^{\pm}[\w,\nu_u]^{\pm}\quad \mathcal H^{N-1}\mbox{-a.e. on }\  J_{u}.
  \end{equation}
 \end{lemma}

\begin{proof} By \eqref{anzellotti-caselles}, the vector field $\z:=u\w$ belongs to $X_\M(\Omega)$. As shown in  \cite[Theorem 3.78]{AFPBook}, $J_{u}$ is a countably $\haus$-rectifiable set oriented by the direction of $\nu_u$. Having in mind the way in which traces of $\w$ are defined over rectifiable sets, it suffices to prove that for any $\Omega'\Subset\Omega$ open with a $C^1$ boundary, then
$$
[\z,\nu_{\Omega'}]^{\pm}=u^\pm[\w,\nu_{\Omega'}]^\pm\quad \mbox{$\mathcal H^{N-1}$-a.e. on \ } \partial\Omega',
$$
which follows directly from Lemma \ref{lemmacaselles}.
\end{proof}

We conclude with two properties of the pairing \eqref{defmeasx144} for bounded $DBV$-functions.

\begin{lemma}\label{l-ms}
  Let $\z\in X_\M(\Omega)$ and let $ u,v\in DBV(\Omega)\cap L^\infty(\Omega) $. Then
  \begin{eqnarray}\label{eq-ms}
 & (u\z, D v)= u(\z,Dv)\ \ \text{as measures}, &
  \\
  \label{f1}
 & (\z,D(uv))=u(\z,Dv)+v(\z,Du)=(u\z,Dv)+(v\z,Du). &
  \end{eqnarray}
\end{lemma}
\begin{proof} The proof of  \eqref{eq-ms} follows line by line the one in  \cite[Proposition 2.3]{MS_acv13} which is based on Lemma \ref{anz} above. A repeated application of Lemma \ref{lemmacaselles} gives
\begin{eqnarray*}
(\z,D(uv)) & = & -uv\dive\z+\dive(uv\z)
\\ &=& -u(\dive(v\z)-(\z,Dv)) + u\dive(v\z) + (v\z,Du)
\\ &=& u(\z,Dv) + (v\z,Du)
\end{eqnarray*}
and \eqref{f1} follows from \eqref{eq-ms}.
\end{proof}

\section{Approximating problems}\label{approx}

We let
$$
|\boeta|_\eps:=\sqrt{|\boeta|^2+\eps^2}
$$
and we note that
\begin{equation}\label{lgp1}
\frac{|\boeta|^2}{|\boeta|_\eps} = \frac{|\boeta|_\eps^2-\eps^2}{|\boeta|_\eps} \ge |\boeta|- \eps.
\end{equation}
For $\eps\in (0,1)$ we consider the following approximating problems:
\begin{equation}
\label{resolvent-eps}
\left\{
\begin{array}{ll}
u - f=\dive\left( (\vare + |u|)^m\frac{\nabla u}{|\nabla u|_\eps}+\eps \nabla u\right) & \mbox{in }\ \Omega
\\
u=g & \mbox{on }\ \partial\Omega.
\end{array}
\right.
\end{equation}
In this section, using standard monotonicity arguments (see for instance \cite{Browder_pnas77} and \cite{ShowalterBook}), we prove the following result.
\begin{lemma}\label{exi-eps}
For any $m\in\re$, any $f\in L^\infty(\Omega)$, and any $g\in L^\infty(\partial\Omega)$, there exists a solution $u_\eps\in H^1(\Omega)\cap L^\infty(\Omega)$ of \eqref{resolvent-eps}  with data $(f,g)$ in the sense that
\begin{equation}\label{weak-eps}
\int_\Omega (u_\eps-f)\varphi = - \int_\Omega \left( (\eps + |u_{\eps}|)^m\frac{\nabla u_\eps}{|\nabla u_\eps|_\eps}+\eps \nabla u_\eps\right)\cdot \nabla \varphi \quad\mbox{ for all $\varphi\in H^1_0(\Omega)$}
\end{equation}
and $u_\eps=g$ on $\partial\Omega$. Furthermore,
\begin{equation}\label{li}
\|u_{\vare}\|_{L^{\infty}(\Omega)}\leq \max\{\|f\|_{L^{\infty}(\Omega)},\|g\|_{L^{\infty}({\partial\Omega})}\}
\end{equation}
and $\ue\ge 0$ if $f\ge 0$ and $g\ge 0$.
\end{lemma}
\begin{proof}
Fix $\delta>0$ and consider the following auxiliary problems:
\begin{equation}
\label{resolvent-eps-delta}
\left\{
\begin{array}{ll}
u - f=\dive\left( (\vare + T_{1/\delta}(|u|))^m\frac{\nabla u}{|\nabla u|_\eps}+\eps \nabla u\right) & \mbox{in }\ \Omega
\\
u=g & \mbox{on }\ \partial\Omega.
\end{array}
\right.
\end{equation}
Fix $\tilde g\in H^1(\Omega)$ such that $\tilde g=g$ on $\partial\Omega$, let $w=u-\tilde g$, and let
$$
A_0(x,w)=w+\tilde g, \quad A_1(x,w,\bxi)=(\vare+ T_{1/\delta}(|w+\tilde g|))^m\frac{\bxi+\nabla \tilde g}{|\bxi+\nabla \tilde g|_\eps}+\eps (\bxi+\nabla \tilde g).
$$
Then \eqref{resolvent-eps-delta} is equivalent to solving
$$
A_0(x,w)-\dive (A_1(x,w,\nabla w))=f, \quad w=0 \ \mbox{ on $\partial\Omega$.}
$$
We note that
\begin{equation}\label{hj}
0\le (\vare+ T_{1/\delta}(|w+\tilde g|))^m \le C
\end{equation}
for some $C>0$ (depending on $\eps$, $\delta$, and $m$). {Existence of solutions follows from, e.g., \cite[Corollary 1]{Browder_pnas77} with $p=2$ in the space $H^1_0(\Omega)$. For its applicability, we need to check:}
\begin{itemize}\item boundedness of $|A_0|$ and $|A_1|$, which follows from
$$
|A_0(x,w)| + |A_1(x,w,\bxi)| \stackrel{\eqref{hj}}\le |w| + \eps |\bxi| +\underbrace{C+|\tilde g|+\eps|\nabla \tilde g|}_{\in L^2(\Omega)};
$$
\item monotonicity of $A_1$, in form of
\begin{eqnarray*}
\left(A_1(x,w,\bxi_1)-A_1(x,w,\bxi_2)\right)\cdot (\bxi_1-\bxi_2) >0 \quad\mbox{for all $\bxi_1\ne\bxi_2$,}
\end{eqnarray*}
which follows from the convexity of the associated Lagrangian,
$$
f(x,w,\bxi)= (\vare+ T_{1/\delta}(|w+\tilde g|))^m|\bxi+\nabla \tilde g|_\eps+\frac{\eps}{2} |\bxi+\nabla \tilde g|^2 \quad (\nabla_{\bxi} f= A_1);
$$
\item coercivity, which follows from
$$
A_1(x,w,\bxi)\cdot\bxi  \stackrel{\eqref{hj}}\ge\frac{\eps}{2}|\bxi|^2 - \underbrace{\left(C |\nabla \tilde g| + \frac{\eps}{2}|\nabla \tilde g|^2\right)}_{\in L^1(\Omega)}.
$$
\end{itemize}
{Uniqueness easily follows by monotonicity.} Therefore \eqref{resolvent-eps-delta} has a unique solution, $u_{\eps,\delta}$. Let   ${k}:=  \max\{\|f\|_{L^{\infty}(\Omega)},\|{ g}\|_{L^{\infty}({\partial\Omega})}\}$, and use $(u_{\eps,\delta} -k)_+:=\max(u_{\eps,\delta} -k,0)$ as test function in \eqref{weak-eps}. We obtain
$$
\int_\Omega (u_{\eps,\delta}-k)_+(u_{\eps,\delta}-f) \le 0,
$$
hence $u_{\eps,\delta} \le k$.  Choosing $\delta < 1/k$, we have $ T_{1/\delta}(|u_{\eps,\delta}|)=|u_{\eps,\delta}|$, hence $u_\eps:=u_{\eps,\delta}$ is a solution to \eqref{resolvent-eps}. Provided $g\ge 0$, choosing $u_-:=\max\{0,-u\}$ as test function in \eqref{weak-eps} we obtain
$$
\int_\Omega (u_{\eps})_- (u_\eps-f) \ge 0,
$$
hence $\ue\ge 0$ if both $f\ge 0$ and $g\geq 0$.
\end{proof}

\section{The singular case}\label{sing}

In this section we study (\ref{pde}) in the singular case, $m<0$. We assume:
\begin{equation}
\label{hp-f-resolvent}
f\in L^\infty(\Omega), \ \ f\geq 0,
\end{equation}
\begin{equation}
\label{hp-g-resolvent}
g\in L^\infty(\partial\Omega), \ g\ge G_0>0.
\end{equation}
Our definition of solution is the following.

\begin{definition}\label{defi1}
Assume $m<0$, \eqref{hp-f-resolvent}, and \eqref{hp-g-resolvent}. A function $u:\Omega\mapsto [0,+\infty)$ is a solution to problem (\ref{pde}) with data $(f,g)$ if
$u\in {BV(\Omega)}\cap L^\infty(\Omega)$, $1/u\in L^\infty(\Omega)$, and there exists a \it gradient-director field $\w\in X_\M(\Omega)$ such that $\|\w\|_{\infty}\leq 1$ and $\z:=u^m\w$ satisfies
\begin{equation}\label{identify-w-new}
{|D \phi(u)|\le (\z,D u)} \quad \mbox{as  measures {in  $\Omega$}, }
\end{equation}
\begin{equation}\label{sm9}
u-f=\dive \z \quad\mbox{in $\mathcal D'(\Omega)$}
\end{equation}
and
\begin{subequations}\label{boundcond}
\begin{equation}
\label{boundcondu<g}
u\le g \quad\mbox{$\mathcal H^{N-1}$-a.e. on $\partial\Omega$,}
\end{equation}
\begin{equation}\label{boundcond=}
 [\z,\nu]=
u^{m} \ \ \  {\rm if \ } u< g \quad\mbox{$\mathcal H^{N-1}$-a.e. on $\partial\Omega$.}
\end{equation}
\end{subequations}
\end{definition}

\begin{remark}\label{reg-zw}
Since $s\mapsto s^m$ is locally Lipschitz in $(0,+\infty)$,
$$
\left.
\begin{array}{l}
u\in BV(\Omega)\cap L^\infty(\Omega;[0,+\infty))
\\
1/u\in L^\infty(\Omega)
\end{array}\right\} \ \Rightarrow \ u^m\in BV(\Omega) \cap L^\infty(\Omega).
$$
In addition, by \eqref{hp-f-resolvent} and \eqref{sm9}, $\dive \z\in L^\infty(\Omega)$; hence $\z \in X(\Omega)$.
\end{remark}

The main result of this section is the following.
\begin{theorem}\label{exi}
Assume $m<0$, \eqref{hp-f-resolvent}, and \eqref{hp-g-resolvent}. Then there exists a unique solution $u$ of \eqref{pde} with data $(f,g)$ in the sense of Definition \ref{defi1}. {In addition, $u\in DBV(\Omega)$ and
\begin{equation}
\label{identify-w-plus}
{|D\phi(u)|}=(\z,Du) \quad\mbox{and} \quad (\w,u)=|Du| \quad\mbox{as measures in $\Omega$.}
\end{equation}
}
\end{theorem}

\subsection{Existence}
The proof of the existence part of Theorem \ref{exi} follows from Lemmas \ref{posit}-\ref{cont} below.

\begin{lemma}[A priori lower bound]\label{posit}
Assume $m<0$,  \eqref{hp-f-resolvent}, and \eqref{hp-g-resolvent}. Positive constants $\alpha$ and $\eps_0$, depending only on $\Omega$ and $G_0$, exist such that for any $\eps<\eps_0$ the solution $u_\eps$ of \eqref{resolvent-eps} with data $(f,g)$ satisfies
$$
u_{\vare}\geq\alpha>0.
$$
\end{lemma}
\begin{proof}
Let $R>0$ be such that $\Omega\subset B(0;R)$. We choose \begin{equation}\label{def-alpha} 0<\alpha<\min\left\{\left(\frac{1}{2^{3-m}(1+R^2)^{3/2}}\right)^\frac{1}{1-m},G_0\right\},\end{equation}
\begin{equation}\label{def-epsilon_0}
 0< \epsilon_0<\min\left\{\frac{G_0-\alpha}{R^2},\frac{\alpha}{2|m|R^2(1+R^2)},\frac{2\alpha}{2+R^2}\right\}.
\end{equation}
We claim that $v_\epsilon(x):=\epsilon \frac{|x|^2}{2}+\alpha$ is a subsolution to \eqref{resolvent-eps} for any $0<\epsilon<\epsilon_0$.
On one hand,
\begin{eqnarray}
\nonumber
\lefteqn{\dive\left((\epsilon+v_\epsilon)^m \frac{\nabla v_\epsilon}{|\nabla v_\epsilon|_\epsilon}+\epsilon \nabla v_\epsilon\right)}
\\ \nonumber &=& (\epsilon+v_\epsilon)^{m-1} \frac{m\epsilon |x|^2}{\sqrt{1+|x|^2}}+(\epsilon+v_\epsilon)^m\left(\frac{1}{(1+|x|^2)^{3/2}} +\frac{N-1}{\sqrt{1+|x|^2}}\right)+\epsilon^2N
\\ \nonumber
&>&(\epsilon+v_\epsilon)^{m-1}\left(\frac{m\epsilon |x|^2}{\sqrt{1+|x|^2}}+\frac{\alpha}{(1+|x|^2)^{3/2}}\right) \quad\mbox{(since $v_\eps>\alpha$)}
\\ \nonumber &>&\frac{(\epsilon+v_\epsilon)^{m-1}}{\sqrt{1+R^2}}\left(m\epsilon R^2+\frac{\alpha}{1+R^2}\right) \quad\mbox{(since $[0,+\infty)\ni s\mapsto \frac{s}{\sqrt{1+s}}$ increases)}
\\ &>& \frac{(\epsilon+v_\epsilon)^{m-1}\alpha}{2(1+R^2)^{3/2}} \quad\mbox{(since $\eps<\frac{\alpha}{2|m|R^2(1+R^2)}$)}.
\label{w2}
\end{eqnarray}
On the other hand,
\begin{equation}\label{w1}
v_\epsilon-f< \eps+v_\eps \le \eps+\alpha+\epsilon \frac{R^2}{2} <2\alpha \quad\mbox{(since $v_\eps<\eps+\frac{R^2}{2}$ and $\eps<\frac{2\alpha}{2+R^2}$).}
\end{equation}
Because of \eqref{w2} and \eqref{w1},
\begin{eqnarray}\label{comp-eps-bulk}
v_\epsilon-f<\dive\left((\epsilon+v_\epsilon)^m \frac{\nabla v_\epsilon}{|\nabla v_\epsilon|_\epsilon}+\epsilon \nabla v_\epsilon\right)
\end{eqnarray}
is implied by
$$
(2\alpha)^{2-m}<\frac{\alpha}{2(1+R^2)^{3/2}},
$$
which is true by \eqref{def-alpha}. The two additional constraints in \eqref{def-alpha} and \eqref{def-epsilon_0} guarantee that $v_\eps\le g$ on $\partial\Omega$. This, together with \eqref{comp-eps-bulk}, implies that $v_\eps\le u_\eps$ in $\Omega$: the argument is analogous, though simpler, to the one used in the proof of Theorem \ref{comp} below, and therefore we omit it.
\end{proof}

\begin{lemma}[Passage to the limit]\label{l-res-weak}
Assume $m<0$, \eqref{hp-f-resolvent}, and \eqref{hp-g-resolvent}. Then there exists $\alpha>0$ and a pair $(u,\w)\in BV(\Omega)\times L^\infty(\Omega;\R^N)$ such that $\|\w\|_\infty \leq 1$,
\begin{equation}\label{lg4}
0<\alpha\le u \le \max\{\|f\|_{L^{\infty}(\Omega)},\|g\|_{L^{\infty}({\partial\Omega})}\},
\end{equation} $\z=u^m\w$ verifies
\begin{equation}\label{lg9}
u-f=\dive \z \quad\mbox{in $\mathcal D'(\Omega)$},
\end{equation}
and
\begin{eqnarray}
  \label{ineq-z-F}|D\phi_F(u)|& \leq & (\z,DF(u))\quad\mbox{as measures}
\\
  \label{ineqboundary}|\phi_F(u)-\phi_F(g)| &\leq&  (F(g)-F(u))[\z,\nu]  \quad\mbox{$\mathcal H^{N-1}$-a.e. on $\partial\Omega$}
\end{eqnarray}
for any nondecreasing $F\in W^{1,\infty}_{loc}((0,+\infty))$, where $\phi_F$ is defined by \eqref{def-phi-F}. In particular,
\begin{eqnarray}
  \label{ineq-z-Id}|D\phi(u)|& \leq & (\z,Du)\quad\mbox{as measures}
\\
  \label{ineqboundary-Id}|\phi(u)-\phi(g)| &\leq&  (g-u)[\z,\nu]  \quad\mbox{$\mathcal H^{N-1}$-a.e. on $\partial\Omega$.}
\end{eqnarray}
\end{lemma}

\begin{proof}
Up to \eqref{lg9}, the proof is rather standard.  Let $u_\eps$ be as in Lemma \ref{exi-eps}. Lemma \ref{posit} and \eqref{li} guarantee that {there exists $\alpha>0$ such that}
\begin{equation}
\label{lg1}
0<\alpha\le \ue\le \max\{\|f\|_{L^{\infty}(\Omega)},\|g\|_{L^{\infty}({\partial\Omega})}\}.
\end{equation}
We define
\begin{equation}
\label{def-z-eps}
\w_\eps:=\frac{\nabla u_\eps}{|\nabla u_\eps|_\eps}\,,\quad \tilde\z_\eps:=  (\vare +u_{\vare})^{m}\w_\eps\,,\quad \z_\eps:= \tilde\z_\eps+\eps \nabla \ue\,.
\end{equation}
Let $\tg\in H^1(\Omega;[G_0,\|g\|_{L^{\infty}({\partial\Omega})}])$ such that $\tg=g$ on $\partial \Omega$. We agree that $\int f\d \mu=\int_\Omega f\d \mu$ and that $\int f =\int f\d x$.  Choosing $\varphi=(u_{\eps}-\tg)$ in \eqref{weak-eps}, we obtain
\begin{eqnarray*}
\int (\ue-\tg)(\ue-f) &=& -\int   (\vare +u_{\vare})^{m}\frac{|\nabla \ue|^2}{|\nabla \ue|_\eps} - \eps\int |\nabla \ue|^2+\int \z_\eps\cdot \nabla\tilde g
\\
 & \stackrel{\eqref{lgp1}} \le& -\int   (\vare +u_{\vare})^{m}(|\nabla \ue|-\eps) - \eps \int |\nabla \ue|^2+\int \z_\eps\cdot \nabla\tilde g.
\end{eqnarray*}
In what follows, $C\ge 1$ denotes a generic constant independent of $\eps\in (0,1)$. In view of \eqref{lg1} we have
\begin{equation}
\label{lg1.5}
C^{-1}\le (\vare +u_{\vare})^{m} \le C
\end{equation}
and
$$
\left|\tilde\z_\eps\right| \stackrel{\eqref{def-z-eps}}\le   (\vare +u_{\vare})^{m} \stackrel{\eqref{lg1.5}}\le C.
$$
Hence
$$
\int \ue^2 +\int |\nabla \ue| + \eps \int |\nabla \ue|^2 \le C \int \left(\eps + |f(\ue-\tg)|+ \ue\tg  + |\nabla\tilde g|+\eps|\nabla u_\eps\cdot\nabla\tg|\right)
$$
and by H\"older and Cauchy-Schwarz inequalities
\begin{equation}\label{lg6}
\int \ue^2 +\int |\nabla \ue| + {\eps} \int |\nabla \ue|^2 \le C \left( 1 + \int f^2 +\int \tg^2
+
\int|\nabla\tg|^2\right).
\end{equation}
By \eqref{lg1} and \eqref{lg6}, along subsequences (not relabeled) we obtain the existence of $u\in BV(\Omega)\cap L^\infty(\Omega)$ and $\w\in L^\infty(\Omega;\rn)$ such that
\begin{eqnarray}
\nonumber
\ue \stackrel{*}\rightharpoonup u && \quad\mbox{in  $BV(\Omega)$ and in $L^\infty(\Omega)$}
\\ \label{lg3+5} \ue \to u && \quad\mbox{$\mathcal L^N$-a.e. and in $L^p(\Omega)$ for all $p<+\infty$}
\\ \label{lgz1} \eps\nabla \ue \rightarrow 0 && \quad\mbox{strongly in $L^2(\Omega;\R^N)$}
\\ \label{w} \w_\eps\stackrel{*}\rightharpoonup \w && \quad\mbox{ in $L^\infty(\Omega;\R^{N})$}
\\ \label{lgz2} \tilde \z_\eps \stackrel{*}\rightharpoonup \z= u^m\w && \quad\mbox{ in $L^\infty(\Omega;\R^{N})$.}
\end{eqnarray}
In addititon, \eqref{lg4} holds. The limits in \eqref{lgz1} and \eqref{lgz2} combine into
\begin{equation}\label{lgz}
\z_\eps\rightharpoonup \z \quad\mbox{ in $L^2(\Omega;\R^{N})$}.
\end{equation}
The bound in \eqref{lg4} follows from \eqref{lg1} and the identity in \eqref{lg9} follows from \eqref{resolvent-eps},  \eqref{lg3+5}, and \eqref{lgz}.

\smallskip

Let $F\in W^{1,\infty}_{loc}((0,+\infty))$ be nondecreasing and $\varphi\in C^\infty(\Omega)$ be nonnegative. Testing \eqref{resolvent-eps} by $\varphi(F(\ue)-F(\tg))$, after an integration by parts we obtain
$$\noindent{\int \varphi F'(\ue) \z_\eps\cdot\nabla \ue}
$$$$=  \int \varphi \z_\eps\cdot \nabla F(\tg) -\int\varphi(F(\ue)-F(\tg))(\ue-f) - \int (F(\ue)-F(\tg)) \z_\eps\cdot \nabla \varphi.$$

On the right-hand side we pass to the limit as $\eps\to 0$ using {\eqref{lg1}}, \eqref{lg3+5} and \eqref{lgz}:
 \begin{equation}\begin{array}{l}
 \lefteqn{\lim_{\eps\to 0} \int \varphi F'(\ue) \z_\eps\cdot\nabla \ue}
\\ \dys = \int \varphi \z\cdot \nabla F(\tg)\ -\int\varphi(F(u)-F(\tg))(u-f) - \int (F(u)-F(\tg)) \z\cdot \nabla \varphi
\\
\dys \stackrel{\eqref{lg9}}= \int \varphi \z\cdot \nabla F(\tg) -\int\varphi(F(u)-F(\tg))\dive\z
- \int (F(u)-F(\tg)) \z\cdot \nabla \varphi.\label{j1}
\end{array}\end{equation}
Note that, by \eqref{lg4}, $F(u)\in BV(\Omega)$. Integrating by parts on the right-hand side of \eqref{j1} and  using Lemma \ref{lemmacaselles}, we see that
\begin{equation}
\label{q1}
\lim_{\eps\to 0} \int \varphi F'(\ue) \z_\eps\cdot\nabla \ue  = \int\varphi\d (\z,DF(u))-\int_{\partial\Omega} \varphi(F(u)-F(g))[\z,\nu]\d\mathcal H^{N-1}.
\end{equation}
Since
$$
\begin{array}{l}
\dys \int \varphi F'(\ue) \z_\eps\cdot\nabla \ue
= \int \varphi F'(\ue)\left(  (\vare +u_{\vare})^{m}\frac{|\nabla \ue|^2}{|\nabla \ue|_\eps}+\eps |\nabla \ue|^2\right)
\\ \\ \dys \stackrel{\eqref{lgp1}}\ge  o_\eps(1) +\int \varphi F'(\ue)   u_{\vare}^{m} |\nabla \ue| \quad\mbox{as $\eps\to 0$},
\end{array}
$$
from \eqref{q1} and  \eqref{def-phi-F}  we derive
\begin{eqnarray*}
\lim_{\eps\to 0} \int \varphi |\nabla\phi_F(\ue)| \leq \int\varphi\d (\z,DF(u))-\int_{\partial\Omega} \varphi(F(u)-F(g))[\z,\nu]\d\mathcal H^{N-1}.
\end{eqnarray*}
Hence, by lower semi-continuity (\cite[Theorem 1]{ACMM_ma10})
\begin{eqnarray*}
\lefteqn{\int \varphi \d |D\phi_F(u)| + \int_{\partial\Omega} \varphi |\phi_F(u)-\phi_F(g)|\d \mathcal H^{N-1}}
\\\nonumber &\le&  \int\varphi\d (\z,DF(u))-\int_{\partial\Omega} \varphi(F(u)-F(g))[\z,\nu]\d\mathcal H^{N-1},
\end{eqnarray*}
which yields \eqref{ineq-z-F} and \eqref{ineqboundary} by the arbitrariness of $\varphi$.
\end{proof}

\begin{lemma}[Trace inequality] \label{lemma36}
Let $u,\w$ and $\z$ be as in Lemma \ref{l-res-weak}. Then $\w\in X_{\M}(\Omega)$,
\begin{equation}
\label{zwboundary}[\z,\nu]= u^m[\w,\nu]\quad\mbox{$\mathcal{H}^{N-1}$-a.e. on $\partial\Omega$,}
\end{equation}
\begin{equation}\label{d1}
u\leq g\quad\mbox{$\mathcal{H}^{N-1}$-a.e. on $\partial\Omega$,}
\end{equation}
and
\begin{equation}
\label{kk}
[\w,\nu]=1 \quad\mbox{$\mathcal{H}^{N-1}$-a.e. on $\partial\Omega\cap\{u<g\}$.}
\end{equation}
\end{lemma}

\begin{proof}
Arguing as in Remark \ref{reg-zw}, we see that both  $u^m$ and $u^{-m}$ belong to $\in BV(\Omega)\cap L^\infty(\Omega)$. Hence, using \eqref{anzellotti-caselles}, we have
$$
\dive \w=\dive(u^{-m}\z)=(u^{-m})^{*}\dive\z +(\z,Du^{-m}),
$$
so that $\w\in X_{\M}(\Omega)$ and  \eqref{zwboundary} follows from \eqref{cas-trace} (applied with $\z$ replaced by $\w$).

By \eqref{ineqboundary-Id}, we have
\begin{equation}\label{b1}
|\phi(u)-\phi(g)|\le (g-u)[\z,\nu] \stackrel{\eqref{zwboundary}}\le  u^m|g-u| \quad\mbox{$\mathcal H^{N-1}$-a.e. on $\partial\Omega$.}
\end{equation}
In particular, $(g-u)[\z,\nu]\ge 0$.
Since $\phi'(s)=s^{m}$ is strictly decreasing, $|\phi(u)-\phi(g)|>|u-g|\min\{ u^m,g^{m}\}$. Therefore,  \eqref{b1} implies that
$$
\min\{ u^m,g^{m}\}|g-u| < (g-u)[\z,\nu] \le  u^m|g-u| \quad\mbox{$\mathcal H^{N-1}$-a.e. on $\partial\Omega$}
$$
whenever $g\ne u$. Hence $g^{m}< u^m$ if $g\ne u$, which means that \eqref{d1} holds.

Let $p>0$. Choosing $F(s)=- \frac{s^{mp}}{p}$ in \eqref{ineqboundary}, {and, therefore,} $\phi_F(s)= \frac{1-s^{m(p+1)}}{p+1}$, then
\begin{equation}\label{boundaryphi'}\left|\frac{ u^{m(p+1)}-g^{m(p+1)}}{p+1}\right|\leq -\frac{g^{mp}-u^{mp}}{p}[\z,\nu]\ \ \text{$\mathcal{H}^{n-1}$-a.e. on}\ \ \partial\Omega\,.\end{equation}
Using the sign properties in \eqref{d1}, we obtain
\begin{eqnarray*}
\frac{p}{p+1}(u^{m(p+1)}-g^{m(p+1)})&\stackrel{\eqref{boundaryphi'}}\le& (u^{mp}-g^{mp})[\z,\nu] \\ \\ & = & (u^{m(p+1)}- u^mg^{mp})\frac{[\z,\nu]}{ u^m}
\\
&\le & ( u^{m(p+1)}-g^{m(p+1)})\frac{[\z,\nu]}{ u^m} \,,
\end{eqnarray*}
$\mathcal{H}^{N-1}$-a.e. on $\partial\Omega$. Therefore
\begin{equation}\label{d2}
\frac{p}{p+1} \le \frac{[\z,\nu]}{ u^m} \stackrel{\eqref{zwboundary}}= [\w,\nu] \mbox{$\mathcal{H}^{N-1}$-a.e.\  on}\  \partial\Omega\cap\{g>u\}.
\end{equation}
Passing to the limit as $p\to +\infty$, \eqref{d2} implies that $[\w,\nu]= 1$ when $u<g$, and \eqref{kk} follows.
\end{proof}

{

The existence part of Theorem \ref{exi} is an immediate consequence of the previous Lemmas.

\begin{proof}[Proof of Theorem \ref{exi} (Existence)]
The pair $(u,\w)$ in Lemma \ref{l-res-weak} has the desired regularity and satisfies \eqref{identify-w-new} and \eqref{sm9} (see \eqref{ineq-z-Id} and \eqref{lg9}). The boundary constraints \eqref{boundcond} follow from Lemma \ref{lemma36}.
\end{proof}

}

{

\subsection{Regularity}

We now prove the regularity part of Theorem \ref{exi}.

\begin{lemma}[{Regularity} of $u$ and identification of $(\w,Du)$] \label{cont}
Let $u$ be a solution to \eqref{pde} in the sense of Definition \ref{defi1}. Then $u\in DBV(\Omega)$ and \eqref{identify-w-plus} holds true.
\end{lemma}

}

\begin{proof} Arguing as in Remark \ref{reg-zw}, we see that $u^m\in BV(\Omega)\cap L^\infty(\Omega)$.
By \cite[Proposition 3.69]{AFPBook}, $J_{u^m}=J_{\phi(u)}=J_u$ and {$\nu_{u^m}=-\nu_u$}. Since $\z\in X (\Omega)$, Lemma \ref{Ammbrdd} implies that
$$
0=(\dive \z)\res J_u=([\z,\nu_u]^+-[\z,\nu_u]^-)\haus\res J_u,
$$
hence
\begin{equation}
\label{fg}
{\Psi:}=[\z,\nu_u]^+=[\z,\nu_u]^-\quad\mbox{$\haus$-a.e. on $J_{u}$}.
\end{equation}
Applying Lemma \ref{lemamb} with $u=u^m$ yields
\begin{equation}\label{e2}
\left|[\z,\nu_{u^m}]^\pm\right| = \left| (u^m)^{\pm}[\w,\nu_u]^\mp\right|\leq {(u^m)^{\pm}}\quad \quad\mbox{$\haus$-a.e. on $J_{u}$}.
\end{equation}
Therefore
\begin{equation}\label{ineqbdrjump}
\left|{\Psi}\right|\stackrel{\eqref{fg},\eqref{e2}} \leq \min\{ (u^{m})^{+}, (u^{m})^{-}\}{=\min\{\phi'(u^+),\phi'(u^-)\}}
\end{equation}
$\haus$-a.e. on $J_u$. On the other hand,
$$
|D\phi(u)|\stackrel{{\eqref{identify-w-new}}}\le (\z,Du)\stackrel{\eqref{anzellotti-caselles}}= -u^*\dive \z+\dive(u\z)\quad {\rm as \ measures. }
$$
Using again that $\z\in X(\Omega)$, this yields \begin{equation}\label{sm} |D^j\phi(u)|\leq  (\dive (u\z))\res_{J_u}.\end{equation}
Applying once more Lemmas \ref{Ammbrdd} and \ref{lemamb}, we obtain from \eqref{sm}:
\begin{eqnarray}
\nonumber
{|\phi(u^+)-\phi(u^{-})| \haus\res J_u} & \stackrel{\eqref{sm}} \leq & \dive(u\z)\haus\res {J_{u}}
 \\ \nonumber
 &=& ([u\z,\nu_{u}]^+-[u\z,\nu_{u}]^-)\haus\res J_{u}
 \\\nonumber &=& {(u^+[\z,\nu_u]^+-u^-[\z,\nu_{u}]^-)\haus\res J_{u}}
 \\ \nonumber
 &{\stackrel{\eqref{fg}}=}& (u^+-u^-){\Psi}\haus\res J_{u}
 \\ \nonumber
 &{\stackrel{\eqref{ineqbdrjump}}\le} & {|u^+-u^-| \min\{\phi'(u^+),\phi'(u^-)\} \haus\res J_{u} }.
\end{eqnarray}
Since $\phi'(s)=s^{m}$ is strictly monotone, we conclude that $\haus(J_u)=0$, hence $u\in DBV(\Omega)$. Consequently, by the chain rule for $BV$-functions,
$$
{u^m |\tilde Du| = |D\phi(u)| {\stackrel{\eqref{identify-w-new}}\le} (\z,Du) \stackrel{\eqref{eq-ms}} = u^m (\w, Du) \le u^m|\tilde Du|}
$$
as measures (recall that  $\tilde D u$  denotes the diffuse part of the gradient of $u$). Therefore $(\z,Du)=|D\phi(u)|$ and $(\w,Du)=|Du|$.
\end{proof}

\subsection{Comparison and Uniqueness}

We have the following contraction principle for solutions to \eqref{pde}.
\begin{theorem}\label{comp}
Assume {$m<0$. Let $f,\overline f$ and $g,\overline{g}$ such that \eqref{hp-f-resolvent}, resp. \eqref{hp-g-resolvent}, hold.}  Let  $u$ and $\overline{u}$ be two solutions of problem {\eqref{pde}} with data $(f,g)$, resp. $(\overline f, \overline g)$.
{If $g\le \overline g$, then}
$$
\into (u-\ou)^{+}\ dx\leq \into (f-\overline{f})^{+}\ dx\,.
$$
In particular, the uniqueness part of Theorem \ref{exi} holds true.
\end{theorem}

\begin{proof}
Let $\w$ and $\overline\w$ be the gradient-director fields associated to $u$, resp. $\overline u$, and let $\z=u^m\w$, $\overline z={\overline u}^m\overline \w$. We know that
\begin{equation}\label{uElliptictproblem1}  u  - f= \dive \z \quad\mbox{and}\quad  \overline{u} - \overline f =  \dive \ \overline \z \quad\mbox{in $L^\infty(\Omega)$}.
\end{equation}
{We also know, by Lemma \ref{cont}, that \eqref{identify-w-plus} holds for both. Hence}
\begin{equation}\label{monotonicity}
(\w-\overline \w, Du-D\overline u) \stackrel{{\eqref{identify-w-plus}}}= |Du|+|D\overline u|  - (\w, D\overline u)-(\overline\w, Du) \ge 0 \quad\mbox{as measures,}
\end{equation}
since $\|\w\|_\infty\le 1$ and $\|\overline \w\|_\infty\le 1$. Multiplying the equations in (\ref{uElliptictproblem1}) by $T_{\epsilon}(u -\overline{u})^+$, applying \eqref{Green}, and subtracting the two equalities we obtain
\begin{eqnarray}
\nonumber
\lefteqn{\int_{\Omega} (u-\overline{u}+\overline f-f) T_\epsilon(u-\overline{u})^+\d x}
\\ \label{E1ENTROP} &=& -\into \d (\z-\overline{\z}, D T_{\vare}(u-\ou)^{+}) + \int_{\partial\Omega}[\z-\overline{\z},\nu]T_{\vare}(u-\ou)^{+}\d\mathcal{H}^{N-1}.
\end{eqnarray}
Let us consider the first term on the right hand side of (\ref{E1ENTROP}). Using the fact that the measure $D(u-\overline u)$ is diffuse, we obtain
\begin{eqnarray}
\nonumber \lefteqn{-\into \d(\z-\overline{\z}, D T_{\vare}(u-\ou)^{+})}
\\ \nonumber &=& -\int_{\{0<u-\ou<\vare\}} \theta (\z-\overline{\z}, D T_{\vare}(u-\ou)^{+})\d |D T_{\vare}(u-\ou)^{+}|
\\ \nonumber
&\stackrel{\eqref{composition}}{=}&-\int_{\{0<u-\ou<\vare\}}\theta (\z-\overline{\z}, D (u-\ou))\d |D (u-\ou)|
\\ \nonumber &=& -\int_{\{0<u-\ou<\vare\}}  \d(\z-\overline{\z}, D (u-\ou))
\\ \nonumber &\stackrel{\eqref{eq-ms}}=& -\int_{\{0<u-\ou<\vare\}}  (u^{m} -\ou^{m})\d(\w, D (u-\ou)) \\\nonumber & & -  \int_{\{0<u-\ou<\vare\}}  \ou^{m}\d(\w-\overline{\w}, D (u-\ou))
\\ \label{p1} &\stackrel{\eqref{monotonicity}}\leq &  - \int_{\{0<u-\ou<\vare\}}  (u^{m}-\ou^{m})\d(\w, D (u-\ou)).
\end{eqnarray}
Since $u,\overline u$ are bounded above and below and the mapping $s\mapsto s^{m}$ is locally Lipschitz, a positive constant $C$, independent of $\eps$, exists such that $|u^m-{\overline u}^m|\le C|u-\ou|$. Using also $\|\w\|_\infty\le 1$ and the fact the measure $D(u-\overline u)$ is diffuse, we see that
\begin{eqnarray}
\nonumber
\left| \int_{\{0<u-\ou<\vare\}}  (u^{m}-\ou^{m})\d(\w, D (u-\ou))\right|
&\leq& C\vare \int_{\{0<u-\ou<\vare\}} \d| D (u-\ou)|
\\ \label{p2} & = & C\eps \int_\Omega \d | D T_\eps(u-\ou)^{+}|.
\end{eqnarray}
By the coarea formula \cite[Theorem 3.40]{AFPBook}, we get
\begin{eqnarray}\nonumber
\int_\Omega \d | D T_\eps(u-\ou)^{+}| &=& \int_{-\infty}^{+\infty} P(\{T_\eps(u-\ou)^{+}>\lambda\})\d \lambda
\\\label{p3} &=& \int_0^\eps P(\{u-\ou>\lambda \})\d \lambda =o_\eps(1) \quad\mbox{as $\eps\to 0$,}
\end{eqnarray}
since $\lambda\mapsto P(\{u-\ou >\lambda \})$ is integrable on $\R$. Inserting \eqref{p1}, \eqref{p2}, and \eqref{p3} into (\ref{E1ENTROP}), dividing  by $\epsilon$, and letting $\epsilon\to 0$, we obtain
$$
\int_{\Omega} (u-\overline{u}+\overline f-f ) {\rm sign}_0^+(u-\overline{u}) \d x \leq \int_{\partial\Omega}([\z, \nu]-[\overline{\z},\nu]){\rm sign}_0^+(u-\overline{u})\d\mathcal{H}^{N-1},
$$
where
$$
\displaystyle {\rm sign}_0^+(r)=\left\{\begin{array}
  {ll} 1 & \quad \mbox{if } r>0 \\ 0 &\quad \mbox{if } r\leq 0.
\end{array}\right.
$$
Since $u \leq g \leq \overline g$ in $\partial\Omega$,
$$
\{ x \in \partial \Omega \ : \ u(x) > \overline{u} (x) \} \subseteq
\{ x \in \partial \Omega \ : \ \overline{g}(x) > \ou (x) \}.
$$
By \eqref{cas-trace} and \eqref{boundcond=}, $[\z,\nu]\leq u^{m}$ and  $[\overline\z,\nu]{=  } \overline{u}^{m}$  $\haus$-a.e. on $\partial\Omega\cap\{u>\overline u\}$. Therefore
$$
\int_{\partial\Omega} \left([\z,\nu]-[\overline \z, \nu] \right) {\rm
sign}_0^+(u-\overline{u})\d\mathcal H^{N-1} $$$$\leq \int_{\partial\Omega} \left(u^{m}-\overline{u}^{m} \right){\rm
sign}_0^+(u-\overline{u})\d\mathcal H^{N-1}  \leq  0\,,
$$
and we conclude that
$$
\int_{\Omega} (u-\overline{u})^{+} \d x \leq  \int_\Omega \left((f-\overline f)^{+}-(f-\overline f)^-\right){\rm
sign}_0^+(u-\overline u) \d x \le \int_\Omega (f-\overline f)^+\d x\,.
$$
\end{proof}

\section{The degenerate case}\label{4}

In this section we analyze the degenerate case of Problem (\ref{pde}), $m>0$. As we already  mentioned,  in contrast with the singular case, for $m>0$ it is natural to allow the data (hence, the solution) to become zero. This reflects into some technical complications in the proofs of both existence and uniqueness, since a priori bounds only guarantee that $T_a^\infty(u)\in BV(\Omega)$ for any $a>0$. Therefore, we will need  some further properties of the space $TBV^{+}$, which are proved in the next subsection.

\subsection{Properties of the space $TBV^{+}(\Omega)$}\label{5.1}

First of all, we argue that
the trace of functions in $TBV^+(\Omega)$ is well defined.

\begin{lemma}\label{L-tbv-trace}
Let $\Omega$ be a bounded open set with Lipschitz boundary and $u\in TBV^+(\Omega)$. Then there exists $u^\Omega\in L^1(\partial\Omega;[0,+\infty))$ such that
\begin{equation}\label{indeed-trace}
\lim_{\rho\to 0} \fint_{\Omega \cap B_\rho(x_0)} |u(x)-u^\Omega(x_0)|\d x =0 \quad\mbox{for $\mathcal H^{N-1}$-a.e. $x_0\in \partial \Omega$.}
\end{equation}
Moreover,
\begin{equation}\label{def-trace}
u^\Omega= \lim_{a\to 0^+} (T^\infty_a(u))^\Omega \quad\mbox{$\mathcal H^{N-1}$-a.e.  in $\partial \Omega$}
\end{equation}
and
\begin{equation}\label{noproblem-trace}
F(u^\Omega) = (F(u))^\Omega \quad\mbox{for all $F\in W^{1,\infty}_{a}$}
\end{equation}
(see \eqref{def-WW} for the definition of $W^{1,\infty}_{a}$).
\end{lemma}
\begin{proof}
Since $u\in TBV^+(\Omega)$, we have $(T^\infty_a(u))^\Omega\in L^1(\partial\Omega)$ for all $a>0$. Of course $0 \leq  (T^\infty_{a'}(u))^\Omega \le (T^\infty_{a''}(u))^\Omega$ for $0<a'<a''$. Hence, by monotone convergence, the point-wise limit $u^\Omega(x)$ in \eqref{def-trace} exists a.e. in $\partial\Omega$ and $u^\Omega \in L^1(\partial\Omega;[0,+\infty))$.  For a.e. $x_0\in \partial\Omega$, we have
$$
\fint_{\Omega \cap B_\rho(x_0)} |u(x)-u^\Omega(x_0)|\d x \leq
\lim_{\rho\to 0} \fint_{\Omega \cap B_\rho(x_0)} |u(x)-T_a^\infty (u(x))|\d x
$$$$
+\lim_{\rho\to 0} \fint_{\Omega \cap B_\rho(x_0)} |T_a^\infty(u(x))-(T_a^\infty (u))^\Omega(x_0)|\d x
$$$$+\lim_{\rho\to 0} \fint_{\Omega \cap B_\rho(x_0)} |(T_a^\infty (u))^\Omega(x_0)-u^\Omega(x_0)|\d x.
$$
Noting that $|u(x)-T_a^\infty (u(x))|= (a-u)\chi_{\{u<a\}}<a$ and recalling \eqref{def-trace}, for any $\eps>0$ we may find $a>0$ such that
\begin{eqnarray*}
\lefteqn{\limsup_{\rho\to 0} \fint_{\Omega \cap B_\rho(x_0)} |u(x)-u^\Omega(x_0)|\d x}
\\ &\le &
\eps +\limsup_{\rho\to 0} \fint_{\Omega \cap B_\rho(x_0)} |T_a^\infty(u(x))-(T_a^\infty (u))^\Omega(x_0)|\d x,
\end{eqnarray*}
hence \eqref{indeed-trace} follows from the arbitrariness of $\eps$ and the definition of trace of $T_a^\infty(u)$. In order to prove \eqref{noproblem-trace}, for $x_0\in \partial \Omega$ we write
\begin{eqnarray*}
|(F(u))^\Omega(x_0)-F(u^\Omega(x_0))| &=& \lim_{\rho\to 0} \fint_{\Omega \cap B_\rho(x_0)}  |(F(u))^\Omega(x_0)-F(u^\Omega(x_0))|\d x
\\ &\le & \lim_{\rho\to 0} \fint_{\Omega \cap B_\rho(x_0)}  |(F(u))^\Omega(x_0)-F(u(x))|\d x
\\ && + \lim_{\rho\to 0} \fint_{\Omega \cap B_\rho(x_0)}  |F(u(x))-F(u^\Omega(x_0))| \d x
\\
& \le & L \lim_{\rho\to 0} \fint_{\Omega \cap B_\rho(x_0)}  |u(x)- u^\Omega(x_0)| \d x
\end{eqnarray*}
and the limit is zero because of \eqref{indeed-trace}.
\end{proof}

In view of \eqref{noproblem-trace}, hereafter we will omit the superscript $\Omega$. The next result is  a version of Lemma \ref{lemamb} for $TBV^+$-functions:

\begin{lemma}\label{lematbvjump}
  Let $u\in TBV^+(\Omega)\cap L^\infty(\Omega)$, $\w\in L^\infty(\Omega; \rn)$ and $\z=u^m\w\in X(\Omega)$. Then
  \begin{itemize}
  \item[(i)] For almost {every $0<a<b\le +\infty$}, $\w\chi_{\{a<u<b\}}\in X_{\mathcal M}(\Omega)$ and \begin{equation}\label{zwboundaryd}[\z,\nu]\chi_{\{{a<u<b}\}}= T_a^b(u)^{m}[\w\chi_{\{{a<u<b}\}},\nu]\ \ \mathcal{H}^{N-1}\text{-a.e. on}\ \  \partial\Omega\,, \end{equation}
  \begin{equation}
  \label{zwboundarydjump}
 {[\z,\nu_{{T_a^b(u)}}]^\pm\chi_{\{{a<u<b}\}}
    }
  = (T_a^b(u)^{m})^\pm[\w\chi_{\{{a<u<b}\}},\nu_{{T_a^b(u)}}]^\pm\,,
  \end{equation}$\mathcal{H}^{N-1}\text{-a.e. on}\ \  {J_{T_a^b(u)}}$\,;
    \item[(ii)]
  {
  \begin{equation}\label{ggf}
  \displaystyle\haus\{x\in \partial\Omega : u(x)=0, [\z,\nu](x)\neq 0\}=0,
  \end{equation}
  }
   \end{itemize}
\end{lemma}

\begin{proof}
Since {$u\in TBV^+(\Omega) \cap L^\infty(\Omega)$, $u^{-m}\chi_{\{a<u<b\}}\in BV(\Omega)\cap L^\infty(\Omega)$ for almost any $0<a<b\le +\infty$. Therefore,} it follows from Lemma \ref{lemmacaselles} {(applied with $u^{-m}\chi_{\{a<u<b\}}$ in place of $u$)} that $\w\chi_{\{{a<u<b}\}}\in X_{\mathcal M}(\Omega)$ and \eqref{zwboundaryd} holds. By the same argument, \eqref{zwboundarydjump} follows immediately from  \eqref{traceproduct}.

\smallskip

Let us prove (ii).
Let $\varphi$ be a non-negative mollifier and $\varphi_\rho(x)=\rho^{-N}\varphi((x-x_0)/\rho)$. Then{, for $\mathcal H^{N-1}$-a.e. $x_0\in \partial \Omega$ we have}
\begin{eqnarray*}
[\z,\nu](x_0) 
&=& C \lim_{\rho\to 0} \rho \int_{\partial\Omega}\varphi_\rho [\z,\nu]\d \haus
\\
&=&  C\lim_{\rho\to 0} \left(\rho \int_{\Omega} \z \cdot\nabla \varphi_\rho\d x + \rho \int_{\Omega}\varphi_\rho \dive \z \d  x\right).
\end{eqnarray*}
The second integral on the r.h.s. vanishes in the limit since $\dive \z\in L^\infty(\Omega)$. For the first one, since $|\z|\le u^m$
{
and $|\nabla \varphi_\rho|\le C \rho^{-N-1}\chi_{B_\rho(x_0)}$, for a.e. $x_0\in \partial\Omega$ we have
\begin{eqnarray*}
\limsup_{\rho\to 0 }\rho \int_{\Omega} \z \cdot\nabla \varphi_\rho\d x & \le &  C \limsup_{\rho\to 0} \fint_{\Omega\cap B_\rho(x_0)} u^m \d x
\\ &\le& C \|u\|_{L^{\infty}(\Omega)}^{m-1} \lim_{\rho\to 0} \fint_{\Omega\cap B_\rho(x_0)} u \d x \ =\  C \|u\|_{L^{\infty}(\Omega)}^{m-1} u(x_0),
\end{eqnarray*}
hence $[\z,\nu](x_0)=0$ $\mathcal H^{N-1}$-a.e. on $\{u=0\}\cap \partial\Omega$.
}
\end{proof}

The last auxiliary result we need shows that, as intuition suggests, in case $u\in DTBV^+(\Omega)$, pairings of the form $(\z,DT(u))$ are oblivious to the values of $\z$ outside supp$(T')$.

\begin{lemma}\label{lemaauxu=0}
  Let $u\in DTBV^+(\Omega)\cap L^\infty(\Omega)$  and $\z\in  X_{\mathcal{M}}(\Omega)$. Then $\z\chi_{\{a<u<b\}}\in X_{\mathcal M} (\Omega)$ for a.e. $0<a<b\le + \infty$ and
  \begin{equation}\label{sd}
  (\z,D {T_a^b(u)})=(\z\chi_{\{ a<u<b\}},D{T_a^b(u)}) \quad\mbox{for a.e. $a<b\leq +\infty$}.
  \end{equation}
\end{lemma}

\begin{proof}
Since $\chi_{\{a<u<b\}}\in BV(\Omega)\cap L^\infty(\Omega)$ for a.e. $a>0$ and a.e. $a<b\leq +\infty$, it follows from Lemma \ref{lemmacaselles} that $\z\chi_{\{a<u<b\}}\in X_{\mathcal M} (\Omega)$ for a.e. $a>0$ and $a<b\leq+\infty$.  We first prove \eqref{sd} for $b=+\infty$, i.e.,
\begin{equation}\label{pp}
  (\z,D{T_a^\infty(u)})=(\z\chi_{\{ u>a\}},D{T_a^\infty(u)}) \quad\mbox{for a.e. $a>0$}.
\end{equation}
We let $\overline T(s):={T_a^\infty(s)-a}$ and we note that
\begin{equation}\label{df}
DT_a^\infty(u)=D\overline T(u), \quad  \overline T(u)=\overline T(u)\chi_{\{u>a\}}.
\end{equation}
Then
\begin{eqnarray*}
(\z\chi_{\{u>a\}},DT_a^\infty(u)) &\stackrel{\eqref{df}}=& (\z\chi_{\{u>a\}},D\overline T(u))
\\ &\stackrel{\eqref{anzellotti-caselles}}=& \dive (\z\chi_{\{u>a\}}\overline T(u))- \overline  T(u) \dive(\z\chi_{\{u>a\}})
\\
&\stackrel{\eqref{anzellotti-caselles},\eqref{df}}=& \dive (\z\overline T(u))- \overline T(u) \dive \z -\overline T(u) (\z,D\chi_{\{u>a\}})
\\
&\stackrel{\eqref{anzellotti-caselles}}=& (\z,D\overline T(u)) -\overline T(u) (\z,D\chi_{\{u>a\}})\,.
\end{eqnarray*}
Note that  $(\z,D\chi_{\{u>a\}})\ll |D\chi_{\{u>a\}}|$ and $\overline T(u)=0$ $|D\chi_{\{u>a\}}|$-a.e. (since $\mathcal H^{N-1}(S_u^*)=0$). Therefore,  $\overline T(u) (\z,D\chi_{\{u>a\}})=0$ and the conclusion follows using again \eqref{df}.

\smallskip

We now prove the statement for a generic $b<+\infty$. The argument is the same, but exploits \eqref{pp}. We note that
\begin{equation}\label{dfd}
T_a^b(s) = T_{a}^{\infty}(T_{b}(s))= T_{b}(T_{a}^{\infty}(s))\quad \mbox{for all $s\geq 0$.}
\end{equation}
Therefore
$$
(\z\chi_{\{a<u<b\}},DT_a^b(u)) \stackrel{\eqref{dfd}}= (\z\chi_{\{u>a\}}\chi_{\{u<b\}},D T_{b}(T_{a}^{\infty}(u)))=(\overline \z \chi_{\{u<b\}}, D \overline T(u)),
$$
where $\overline \z = \z \chi_{\{u>a\}}$  $\overline  T(u)=T_{b}(T_{a}^{\infty}(u))-b$. Noting that $\overline T(u)=\overline T(u)\chi_{\{u<b\}}$ and arguing exactly as above, we obtain
$$
(\z\chi_{\{a<u<b\}},DT_a^b(u))= (\overline \z,D\overline T(u)) -\overline T(u) (\overline \z,D\chi_{\{u<b\}})= (\overline \z,D\overline T(u)),
$$
where in the last equality we have used that $(\overline \z,D\chi_{\{u<b\}})\ll |D\chi_{\{u<b\}}|$ and that $\overline T(u)=0$ $|D\chi_{\{u<b\}}|$-a.e. (here we use again that $\mathcal H^{N-1}(S_u^*)=0$).  Therefore, recalling the definition of $\overline \z$ and $\overline T$,
$$
(\z\chi_{\{a<u<b\}},DT_a^b(u))=(\z\chi_{\{u>a\}},D T_a^\infty(T_b(u))) \stackrel{\eqref{dfd}}= (\z,D T_a^b(u)).
$$
\end{proof}

{
\subsection{Existence}
We can now look at the  existence of a solution to \eqref{pde} in the case $m>0$.  We assume:
\begin{equation}
\label{hp-f-resolventd}
{0\leq f\in L^\infty(\Omega)\,,\quad
0\leq g\in L^\infty(\partial\Omega).}
\end{equation}

We introduce the following notion of solution. }

\begin{definition}\label{defi1d}
Assume $m>0$ and \eqref{hp-f-resolventd}. A function {$u:\Omega\to [0,+\infty{)}$} is a solution {of} problem {(\ref{pde})} {with data $(f,g)$} if $u\in {TBV^+(\Omega)}\cap L^\infty(\Omega)$ {and} there exist $\w\in L^\infty(\Omega;\rn)$ such that $\|\w\|_{\infty}\leq 1$ and $\z:=u^m\w\in X(\Omega)$ satisfies
{
\begin{equation}\label{identify-wd-new}
|D \phi({T}(u))| \le (\z,D T(u)) \quad \mbox{as measures for any } T\in\mathcal T,
\end{equation}
}
\begin{equation}\label{sm9d}
u-f=\dive \z \quad\mbox{in $\mathcal D'(\Omega)$}
\end{equation}
and
\begin{subequations}\label{boundcondd}
\begin{equation}\label{boundconddu>g}
\mbox{$u\ge g$ $\haus$-a.e. on $\partial \Omega$}\,,
\end{equation}
\begin{equation}\label{boundcondd=}  [\z,\nu] =
-u^{m} \ \  {\rm if \ } {u > g} \quad\mbox{$\mathcal H^{N-1}$-a.e. on $\partial\Omega$.}
\end{equation}
\end{subequations}
\end{definition}

\begin{remark} \label{vb}
{
The} boundary conditions are consistent with  those used  in \cite{ACMM_ma10} for equation \eqref{pmrhe} with $m=1$.
\end{remark}

Definition \ref{defi1d} differs from Definition \ref{defi1}
{
since we allow data (and therefore solutions) {to} become zero: since the equation degenerates, we have little control at $u=0$ and we need to use truncation functions.} For data which are bounded away from zero this new formulation is not needed and well-posedness can be obtained as in the previous section with minor modifications. Indeed, if there exists $C>0$ such that $C\leq f(x)$ for a.e.  $x\in\Omega$ and $C\leq g(x)$, for a. e.  $x\in\partial\Omega$,  it is straightforward to see that $v\equiv C$ is a subsolution to \eqref{resolvent-eps}. Therefore the approximate solutions, whence the limiting solutions obtained in Lemma \ref{l-res-weakd} below, are strictly positive.

\smallskip

The main result of this section is the following.

\begin{theorem}\label{exid}
Assume $m>0$ and \eqref{hp-f-resolventd}. Then there exists a unique solution $u$ of {\eqref{pde}} {with data $(f,g)$ in the sense of Definition \ref{defi1d}}. {Furthermore, $u\in DTBV^+(\Omega)$,
\begin{subequations}\label{em-tot}
\begin{equation}\label{em1}
{({\w},DT_a^b(u))=|D T_a^b(u)|} \quad\mbox{for a.e. $0<a<b\leq +\infty$,}
\end{equation}
and
\begin{equation}\label{em2}
({\z},D T_a^b(u))=\left|D{\phi(T_a^b(u))}\right| \quad\mbox{for a.e. $0<a<b\leq +\infty$.}
\end{equation}
\end{subequations}
}
\end{theorem}

{In proving existence of a solution, we will follow the arguments used in the singular case highlighting only the main differences, which are related to the need of using truncation functions.}

\begin{lemma}\label{l-res-weakd}
Assume $m>0$ and \eqref{hp-f-resolventd}. Then there exists a pair $(u,\w)$ such that $u\in TBV^{+}(\Omega)\cap L^\infty(\Omega)$ and $\w\in {L^\infty(\Omega;\rn)}$ with
$\|\w\|_\infty \leq 1$, such that {$\z=u^m\w \, \in X(\Omega)$}, and
\begin{equation}\label{lg9d}
u-f=\dive \z \quad\mbox{in $\mathcal D'(\Omega)$.}
\end{equation}
Furthermore,
\begin{equation}
  \label{ineq-z-Fd}|D\phi_{F}(T(u))|  \leq  (\z,DF(T(u))) \quad\mbox{as measures,}
\end{equation}
and
 \begin{equation}
  \label{ineqboundaryd} |\phi_{F}(T(g)) - \phi_{F}(T(u)) |  \leq  (F(T(g))- F(T(u)))[\z,\nu]\quad\mbox{$\mathcal H^{N-1}$-a.e. on $\partial\Omega$}
\end{equation}
for any $T\in\mathcal{T}$ and any nondecreasing $F\in W^{1,\infty}_{loc}((0,+\infty))$, with $\phi_F$ as in \eqref{def-phi-F}.
\end{lemma}

\begin{proof} Arguments are analogous to those of Lemma {\ref{l-res-weak}}. Let $\tg$ be a function in $H^1(\Omega;[0,\|{ g\|_{L^{\infty}({\partial\Omega})}])}$ such that $\tg=g$ on $\partial \Omega$.  Again, for simplicity, we agree that $\int f\d \mu=\int_\Omega f\d \mu$, $\int f =\int f\d x$, and $C\ge 1$ denotes a generic constant independent of $\eps\in (0,1)$. Let $u_\eps{\in H^{1}(\Omega)\cap L^\infty(\Omega)}$ be a solution of \eqref{resolvent-eps} as given by Lemma \ref{exi-eps}. We recall that
\begin{equation}
\label{lid}
{0\le u_{\vare}} \leq \max\{\|f\|_{L^{\infty}(\Omega)},\|g\|_{L^{\infty}({\partial\Omega})}\}.
\end{equation}
Testing the equation \eqref{resolvent-eps} by $\ue-\tg$ and using that $\eps^m\le (\vare+\ue)^{m}\leq C$, we get
$$
\int \ue^2 +\int \ue^m|\nabla \ue| + \frac{\eps}{2} \int |\nabla \ue|^2 \leq C\left( 1 + \int f^2 +\int \tg^2+\int|\nabla\tg|^2\right)
$$
and since
$$
\dys \into |\nabla T_a^\infty(\ue)| \leq \int_{\{a\leq \ue\}} |\nabla \ue|
\leq  \frac{1}{a^{m}}\int_{\{a\leq \ue\}}\ue^{m}|\nabla \ue|\leq  \frac{1}{a^{m}}\into\ue^{m}|\nabla \ue|\,,
$$
we conclude that
\begin{equation}\label{lg6d}
\int \ue^2 + {a^m}\int |\nabla T_{a}^\infty(\ue)| {+ \int |\nabla (\ue^{m+1})|}+ {\eps} \int |\nabla \ue|^2 \leq C.
\end{equation}
{Because of \eqref{lg6d} and \eqref{lid}, there exist $u\in TBV^{+}(\Omega)\cap L^\infty(\Omega)$ and $\w\in L^\infty(\Omega;\rn)$ such that (up to subsequences, not relabeled)}
\begin{eqnarray}
\nonumber
\ue \stackrel{*}\rightharpoonup u && \quad\mbox{in  $TBV^{+}(\Omega)$ and in $L^\infty(\Omega)$}
\\
\nonumber
\ue \to u && \quad\mbox{$\mathcal L^N$-a.e. and in $L^p(\Omega)$ for all $p<+\infty$}
\\
\nonumber
{\ue^{m+1} \stackrel{*}\rightharpoonup u^{m+1}} && {\quad\mbox{in  $BV(\Omega)$}}
\\
\label{lgz1d} \eps\nabla \ue \rightarrow 0 && \quad\mbox{strongly in $L^2(\Omega;\R^{N})$}
\\
\nonumber
{\nabla u_\eps/|\nabla u_\eps|_\eps=:} \w_\eps\stackrel{*}\rightharpoonup \w && \quad\mbox{ in $L^\infty(\Omega;\R^{N})$}
\\ \label{lgz2d} {(\vare +u_{\vare})^{m}\w_\eps=:} \tilde \z_\eps \stackrel{*}\rightharpoonup \z= u^m\w && \quad\mbox{ in $L^\infty(\Omega;\R^{N})$},
\end{eqnarray}
and \eqref{lgz1d} and \eqref{lgz2d} combine into
\begin{equation*}
{\tilde\z_\eps+\eps \nabla \ue=:}\z_\eps\rightharpoonup \z \quad\mbox{ in $L^2(\Omega;\rn)$}\,.
\end{equation*}
Passing to the limit as $\eps\to 0$ in the approximating equations we obtain \eqref{lg9d}.

\smallskip

The proof of (\ref{ineq-z-Fd}) and (\ref{ineqboundaryd}) is a straightforward adaptation of that of (\ref{ineq-z-F}) and (\ref{ineqboundary}), testing \eqref{resolvent-eps} by $\varphi(F(T(\ue))-F(T(\tg)))$ with $0\leq\varphi\in C^\infty(\Omega)$. Therefore we omit the details.
\end{proof}

We have the following:

\begin{lemma}\label{lematp}  Assume $m>0$ and \eqref{hp-f-resolventd}. Let $u,\w$ be as in Lemma \ref{l-res-weakd}. Then
\begin{equation}\label{gh}
u\ge g  \quad \mbox{$\haus$-a.e. on $\partial \Omega$}.
\end{equation}
Furthermore,
\begin{equation}\label{boundcondd1}  {[\z,\nu] = -u^{m}} \quad {\rm if \ } {u > g} \quad\mbox{$\mathcal H^{N-1}$-a.e. on $\partial\Omega$.}
\end{equation}
\end{lemma}

\begin{proof} The proof is analogous to the one of Lemma \ref{lemma36}, hence we only show the main differences. For notational convenience, we let $T=T_a^b\in \mathcal T$. For \eqref{gh}, applying {\eqref{ineqboundaryd}} with $F(s)=s$, we see that
\begin{equation}
\label{fgh}
\frac{1}{m+1}|(T(g))^{m+1} - (T(u))^{m+1}| \leq (T(g)- T(u))[\z,\nu] {\quad\mbox{$\mathcal H^{N-1}$-a.e. on $\partial\Omega$.}}
\end{equation}

We now argue for a fixed $x\in \partial \Omega$ and up to  $\mathcal H^{N-1}$-negligible sets. If $u(x)=0$ at some point $x\in\partial\Omega$, it follows from \eqref{ggf} that $[\z,\nu]=0$. Hence \eqref{fgh} implies that {$(T(g))^{m+1} = a^{m+1}$} for all $T\in\mathcal T$: therefore $g(x)=0$ and \eqref{gh} holds. If instead $u(x)\neq 0$, let $a$ and $b$ such that $a<u(x)<b$.  We have
\begin{eqnarray}\nonumber
\frac{1}{m+1}|T(g)^{m+1}- T(u)^{m+1}| &\le& (T(g)- T(u))[\z,\nu]
\\\label{as}
& \stackrel{\eqref{zwboundaryd}}\le & T(u)^{m}|T(g)- T(u)|,
\end{eqnarray}
which implies
\eqref{gh} arguing as in the proof of \eqref{d1}.

In order to prove \eqref{boundcondd1}, let $F(s)=\frac{s^{mp}}{p}$  and  let us  fix  $x\in\partial \Omega$ such that \eqref{ineqboundaryd} holds true. Then ${\phi_{F}(T(s))}= \frac{T(s)^{m(p+1)}{-1}}{p+1}$ and we have
\begin{equation}
\label{boundaryphi'd}
\left|\frac{ (T(u))^{m(p+1)}-(T(g))^{m(p+1)}}{p+1}\right|\leq \frac{((T(g))^{mp}- (T(u))^{mp})}{p}[\z,\nu].
\end{equation}
The rest of the proof is similar to that of \eqref{kk} and we omit it.
\end{proof}

The existence part of Theorem \ref{exid} is an immediate consequence of the previous lemmas:

\begin{proof}[Proof of Theorem \ref{exid}, existence]
{Lemma \ref{l-res-weakd}} gives the existence of a function $u\in {TBV^{+}(\Omega)}\cap L^{\infty}(\Omega)$, and $\w\in L^{\infty}(\Omega;\rn)$ with $\|\w\|_{\infty}\leq 1$ such that $\z=u^{m}\w\in X(\Omega)$ and {\eqref{identify-wd-new}} and \eqref{sm9d} are satisfied. The boundary datum $g$ is achieved in the sense of Definition \ref{defi1d} thanks to Lemma \ref{lematp}.
\end{proof}

{

\subsection{Regularity}

In the next two Lemmas, we show that any solution to \eqref{pde} in the sense of Definition \ref{defi1d} has the additional regularity properties stated in Theorem \ref{exid}. First we show that, as in the singular case, solutions' gradients have no jump part.

}

\begin{lemma}\label{4.8}
Assume $m>0$ and \eqref{hp-f-resolventd}.
Let $u\in TBV^{+}(\Omega)\cap L^\infty(\Omega)$ and $\w\in {L^\infty(\Omega)}$ be such that $\z= u^m\w\in X(\Omega)$, $\|\w\|_\infty\leq 1$ and
\begin{equation}
  \label{identify-zd} |D{^j\phi(T(u))}|\leq (\z,DT(u)) {\res J_{T(u)}}\quad {\rm  as \ measures \ }
\end{equation}
for any $T\in\mathcal{T}^{\infty}$.
Then $u\in DTBV^{+}(\Omega)$.
\end{lemma}
\begin{proof}
 Let {$T=T_a^\infty$, $a>0$,} and recall that $\phi(s)=\frac{s^{m+1}}{m+1}$. Note that $J_{T(u)}=J_{\phi(T(u))}$ and {$\nu_{T(u)}=\nu_{T(u)^{m}}$} on $J_{T(u)}$. Since $\z \in X(\Omega)$,
\begin{equation}\label{bnm}
(\dive \z) \res{{J_{T(u)}}=0}
\end{equation}
and
\begin{equation}\label{bnn}
[\z,\nu_{T(u)}]^+ \stackrel{\eqref{Tresd}}= [\z,\nu_{T(u)}]^-=:{\Psi} \quad\mbox{$\mathcal H^{N-1}$-a.e. on $J_{T(u)}$.}
\end{equation}
{Therefore, by \eqref{zwboundarydjump}, {for almost every $a>0$},
\begin{equation}\label{crt}
|\Psi|\chi_{\{u>a\}} \le \min \{(T(u)^{m})^+,(T(u)^{m})^-\} \quad\mbox{$\mathcal H^{N-1}$-a.e. on $J_{T(u)}$.}
\end{equation}
}
We have
\begin{eqnarray}
 \nonumber |D^j\phi(T(u))|& {\stackrel{\eqref{identify-zd}}\leq} & (\z,DT(u))\res{J_{T(u)}}
 \\ \nonumber
 &\stackrel{\eqref{anzellotti-caselles}}=&  (-{T(u)}^{*}\dive \z+\dive(T(u)\z))\res{J_{T(u)}}
 \\ \label{identify-zd+}
 &\stackrel{\eqref{bnm}}{=}&  \dive(T(u)\z) \res{J_{T(u)}}
\end{eqnarray}
Arguing as in the proof of Lemma \ref{cont}, Lemmas \ref{Ammbrdd}-\ref{lemamb} and \eqref{identify-zd+}
imply that $T(u)^+=T(u)^-$. Therefore, $0=\mathcal H^{N-1}(J_{T(u)})$ ${=\mathcal H^{N-1}(S_{T(u)})}$ {for almost every $a>0$}: by Lemma \ref{tbvjump}, $\mathcal H^{N-1}(S_u^*)=0$ and the proof is complete.
\end{proof}

\begin{lemma}\label{equalmeasured} Let $u\in DTBV^+(\Omega)\cap L^\infty(\Omega)$ and $\w\in X_\M(\Omega)$ be such that $\z= u^m\w\in X(\Omega)$, $\|\w\|_{\infty}\leq 1$ and {\eqref{identify-wd-new}} holds. Then {\eqref{em-tot} holds}.
\end{lemma}

\begin{proof} Letting $T=T_a^b$, {we notice that
\begin{equation}\label{sdf}
\mbox{$u=T(u)\ $ $|DT(u)|$-a.e.}\quad\mbox{and}\quad (\w\chi_{\{a<u<b\}},D (T(u)))\ll |DT(u)|.
\end{equation}
}
Therefore
\begin{eqnarray*}
\left|D {\phi(T(u))}\right| & \stackrel{{\eqref{identify-wd-new}}}\leq &  (\z,DT(u))
\stackrel{\eqref{sd}}= (\z\chi_{\{a<u<b\}},DT(u))
\\ \nonumber &\stackrel{\eqref{eq-ms}}= & u^m (\w\chi_{\{a<u<b\}},DT(u))
\\ \nonumber & \stackrel{\eqref{sdf}}= & T(u)^m (\w\chi_{\{a<u<b\}},DT(u))
\\ & = & T(u)^m \theta(\w\chi_{\{a<u<b\}},D T(u)) |DT(u)|
\\ \nonumber &=& \theta(\w\chi_{\{a<u<b\}},D T(u))\left|{D\phi(T(u))}\right| \\\nonumber & \le & \left|{D\phi(T(u))}\right|,
\end{eqnarray*}where in the last equality we have used the fact that $\mathcal H^{N-1}(S_u^*)=0$. Hence \eqref{em2} holds and  $\theta(\w\chi_{\{a<u<b\}},D (T(u))) =1$ $|DT(u)|$-a.e., whence \eqref{em1}.
\end{proof}

\subsection{Comparison and uniqueness}
The uniqueness part of Theorem \ref{exid} is an immediate consequence of the following comparison principle.

\begin{theorem}\label{UniqEllipticd} Assume $m>0$ and $f,\overline f$ and $g,\overline{g}$ such that \eqref{hp-f-resolventd} holds.  Let  $u, \overline{u}\in DTBV^+(\Omega)\cap L^\infty(\Omega)$ be two solutions of problem \eqref{pde} with data $(f,g)$, resp. $(\overline f, \overline g)$. If $g\le \overline g$, then
$$
\int_\Omega (u - \overline{u})^+ \leq \int_\Omega (f - \overline{f})^+.
$$
{In particular, the uniqueness part of Theorem \ref{exid} holds true.}
\end{theorem}

\begin{proof}

Let $\w$, resp. $\overline \w$, and $\z$, resp. $\overline \z$, be as in Definition \ref{defi1d} for $u$, resp. $\overline u$. In particular,
\begin{equation}
\label{obv}
u-f=\dive \z \quad\mbox{and}\quad \overline u-\bar f=\dive \overline \z \quad\mbox{in $L^\infty(\Omega)$.}
\end{equation}
{In addition, it follows from Lemmas \ref{4.8} and \ref{equalmeasured} that $u,\overline u\in DTBV^+(\Omega)$ and that \eqref{em-tot} holds for both pairs. Consequently, \eqref{em1} and Lemma \ref{lemaauxu=0} imply that
  \begin{equation}\label{sdsd}
  \left. \begin{array}{r} |D {T_a^b(u)}|=(\w\chi_{\{ a<u<b\}},D{T_a^b(u)})
  \\ |D {T_a^b(\overline u)}|=(\overline\w\chi_{\{ a<\overline u<b\}},D{T_a^b(\overline u)})
  \end{array}\right\}\quad\mbox{for a.e. $0<a<b\leq +\infty$}.
  \end{equation}
}
Given $b > a > 0$, we let
\newcommand{\Tae}{T_{a,\eps}(u,\bar u)}
\begin{equation}\label{ud-def}
T(r):=T_{a}^{b}(r)-a \quad\mbox{and}\quad \Tae:=T_\eps((T_a^\infty(u)-T_a^\infty(\overline u))_+).
\end{equation}
We multiply {\eqref{obv}$_1$} by $T(u)\Tae$ and {\eqref{obv}$_2$} by $T(\overline u)\Tae$, integrate by parts, and subtract both identities. Then,
\begin{eqnarray}
\nonumber \lefteqn{\int_{\Omega}
((u-f)T(u)-(\overline{u}-\overline f)T(\overline{u}))
\Tae \d x}
\\ \nonumber &=& -\int_{\Omega} \d((\z,
D(T(u)\Tae)- (\overline{\z}, D(T(\overline{u})\Tae)
\\ \nonumber && +\int_{\partial\Omega}\Tae\left(T(u)[\z,\nu]-T(\overline u)[\overline \z,\nu]\right) \d\mathcal H^{N-1}
\\  \nonumber &\stackrel{\eqref{f1}} =&
 -\int_{\Omega} \Tae \d ((\z,D (T (u)+a))-(\overline \z, DT(\overline u)))
 \\ \nonumber && -\int_\Omega \left( T(u)\d(\z,D\Tae)- T(\overline u) \d(\overline \z,D\Tae)\right)
\\ \nonumber && +
\int_{\partial\Omega}\Tae\left(T(u)[\z,\nu]-T(\overline u)[\overline \z,\nu] \right) d\mathcal H^{N-1}
\\ \label{i} &=:&I_1+I_2+I_3.
\end{eqnarray}
As to $I_1$, we have
\begin{equation}\label{i1}
I_1 \stackrel{{\eqref{em2}}}\le \int_{\Omega} \Tae \d (\overline \z, DT(\overline u))\stackrel{\eqref{ud-def}}\le \eps\into \d (\overline \z, DT(\overline u)).
\end{equation}
As to $I_2$, by Lemma \ref{l-ms} and since $\z T(u)=\w u^mT(u)\chi_{\{u>a\}}$, we have
\begin{eqnarray*}
T(u)(\z,D\Tae)
& = & (\z T(u),D\Tae) \\\nonumber
&=&T(u)u^m(\w\chi_{\{u>a\}},D\Tae).
\end{eqnarray*}
Similarly,
$$
T(\overline u)(\overline\z,D\Tae)
=T(\overline u)\overline u^m(\overline \w\chi_{\{\overline u>a\}},D\Tae).
$$
Then, since $\mathcal H^{N-1}(J_{T(u)u^m}^*)=0$ and  $(\overline \w\chi_{\{\overline u>a\}},D\Tae) {\ll} |D\Tae|$, we can add and subtract $T(u)u^m\d(\overline \w\chi_{\{\overline u>a\}},D\Tae)$ to $I_2$  to get
\begin{eqnarray}
\nonumber I_2 &= &
 -\int_\Omega T(u)u^m\d(\w\chi_{\{u>a\}}-\overline \w\chi_{\{\overline u>a\}},D\Tae)
\\ \nonumber && + \int_\Omega (T(u)u^m-T(\overline u)\overline u^m)\d (\overline \w\chi_{\{\overline u>a\}},D\Tae)
\\ \label{i2}
& =:&I_{2,1} + I_{2,2}.
\end{eqnarray}

As to $I_{2,1}$, {using Lemma \ref{lematbvjump} we deduce that both $\w\chi_{\{u>a\}}$ and $\overline \w\chi_{\{\overline u>a\}}$ belong to  $X_{\mathcal{M}}(\Omega)$}, so that  we have
\begin{eqnarray*}
\lefteqn{(\w\chi_{\{u>a\}}-\overline \w\chi_{\{\overline u>a\}},D\Tae)}
\\
&\stackrel{\eqref{composition2}}=& \chi_{\{0<T_a^\infty(u)-T_a^\infty(\bar u)<\eps\}} (\w\chi_{\{u>a\}}-\overline \w\chi_{\{\overline u>a\}},D(T_a^\infty(u)-T_a^\infty(\bar u)))
\end{eqnarray*}
and
\begin{eqnarray*}
\lefteqn{(\w\chi_{\{u>a\}}-\overline \w\chi_{\{\overline u>a\}},D(T_a^\infty(u)-T_a^\infty(\bar u)))}
\\
&\stackrel{{\eqref{sdsd}}}=& |DT_a^\infty(u)|+|DT_a^\infty(\bar u)|-(\w\chi_{\{u>a\}},DT_a^\infty(u))
-(\overline \w\chi_{\{\overline u>a\}},DT_a^\infty(\bar u))
\\ &\ge& 0 \quad\mbox{(since $\|\w\|\le 1$ and $\|\bar \w\|\le 1$)},
\end{eqnarray*}
hence
\begin{equation}\label{i21}
I_{2,1}\le 0.
\end{equation}
As to $I_{2,2}$, {again in view of Lemma  \ref{lematbvjump}}, we have
\begin{eqnarray*}
\lefteqn{\left|(\overline \w\chi_{\{\overline u>a\}},D\Tae)\right|}
\\ &\stackrel{\eqref{composition2}}= & \chi_{\{0<T_a^\infty(u)-T_a^\infty(\bar u)<\eps\}} \left|(\overline \w\chi_{\{\overline u>a\}},D(T_a^\infty(u)-T_a^\infty(\bar u)))\right| \\ & \le & \chi_{\{0<T_a^\infty(u)-T_a^\infty(\bar u)<\eps\}} |D(T_a^\infty(u)-T_a^\infty(\bar u))|
\end{eqnarray*}
and
\begin{eqnarray*}
\lefteqn{\chi_{\{0<T_a^\infty(u)-T_a^\infty(\bar u)<\eps\}}
|T(u)u^m-T(\overline u)\overline u^m|}
\\
&=& \left\{\begin{array}{ll} |u^{m}(u-a)-\bar u^{m}(\bar u-a)| & \mbox{ if }\ u>a,\ \bar u>a, \ 0<u-\bar u<\eps
\\
|u^{m}(u-a)| &\mbox{ if }\  u>a ,\ \bar u<a,\ 0<u-a<\eps
\\
0 & \mbox{otherwise}.
\end{array}\right.
\\ &\le& C(a)\eps.
\end{eqnarray*}
Therefore, by the coarea formula,
\begin{equation}\label{i22}
I_{2,2} \le C(a)\eps \int_\Omega  \chi_{\{0<T_a^\infty(u)-T_a^\infty(\bar u)<\eps\}} \d|D(T_a^\infty(u)-T_a^\infty(\bar u))| =\eps o_\eps(1)
\end{equation}
as $\eps\to 0$. Combining \eqref{i1}, \eqref{i2}, \eqref{i21}, and \eqref{i22}, dividing \eqref{i} by $\eps$, and passing to the limit as $\eps\to 0$, we obtain
\begin{eqnarray}
\nonumber \lefteqn{\int_{\Omega}
((u-f)T(u)-(\overline{u}-\overline f)T(\overline{u}))
{\rm sign }_0^+(T_a^\infty(u)-T_a^\infty(\bar u)) \d x}
\\ \nonumber &\le &
 \int_{\Omega} (\overline \z, DT(\overline u)))
 \\ \label{ib} &+& \int_{\partial\Omega}{\rm sign }_0^+(T_a^\infty(u)-T_a^\infty(\bar u))\left(T(u)[\z,\nu]-T(\overline u)[\overline \z,\nu] \right) d\mathcal H^{N-1}.
\end{eqnarray}
The boundary integral in \eqref{ib} is non-positive: indeed, $T_a^\infty(u)-T_a^\infty(\bar u)>0$ implies $u>a$ and $u>\bar u$, and $u>\bar u$ implies $u>g$ since $g\le \bar g \le \bar u$. Therefore
\begin{eqnarray*}
\lefteqn{
T(u)[\z,\nu]-T(\overline u)[\overline \z,\nu] { \stackrel{\eqref{ud-def}}=  T(u)[\z,\nu]\chi_{\{u>a\}}-T(\overline u)[\overline \z,\nu]\chi_{\{\overline u>a\}}}
}
\\ &\stackrel{\eqref{zwboundaryd}}=&
T(u)(T_a^{\infty}(u))^{m}[\w\chi_{\{a<u<b\}},\nu]-T(\overline u)(T_a^{\infty}(\overline u))^{m}[\overline \w\chi_{\{a<u<b\}},\nu]
\\
&\stackrel{\eqref{boundcondd=},\ {g<u}}\le & -T(u)(T_a^{\infty}(u))^{m}+ T(\overline u)(T_a^{\infty}(\overline u))^{m} \ \stackrel{u>\bar u}\le \ 0.
\end{eqnarray*}
Hence, dividing \eqref{ib} by $b$ and passing to the limit as $a\to 0$ and $b \to 0$ (in this order), we obtain
\begin{equation}\label{ic}
 \int_{\Omega} ((u-f)\chi_{\{u>0\}}-(\overline{u}-\overline f)\chi_{\{\bar u>0\}})
{\rm sign }_0^+(u- \bar u) \d x
\le  \lim_{b\to 0}\frac{1}{b}\left(\lim_{a\to 0}\int_{\Omega} (\overline \z, DT(\overline u))\right).
\end{equation}
Let now $\tilde T(u)=T_a^b(u)-b$. We notice that
\begin{equation}\label{ggg}
D\tilde T(\bar u)=DT(\bar u), \quad \frac{1}{b}\tilde T(s)\stackrel{a\to 0}\rightarrow \frac{1}{b}( T_0^b(s)-b) \stackrel{b\to 0}\rightarrow -\chi_{\{s\le 0\}}
\end{equation}
and that
\begin{eqnarray*}
\nonumber
0 &\stackrel{{\eqref{em2}}}\le & \int_\Omega \d (\bar\z, DT(\bar u))\stackrel{\eqref{ggg}_1}=\int_\Omega \d (\bar\z, D\tilde T(\bar u))
\\ \nonumber
&\stackrel{\eqref{defmeasx144},\eqref{obv}_2}= & \int_{\partial\Omega} \tilde T(\bar u) [\bar \z,\nu] \d\mathcal H^{n-1}-\int_\Omega \tilde T(\bar u) (\bar u-\bar f).
\end{eqnarray*}
Therefore
\begin{eqnarray}
\nonumber
0 &\stackrel{{\eqref{em2}}}\le & \lim_{b\to 0}\frac{1}{b}\left(\lim_{a\to 0} \int_\Omega \d (\bar\z, DT(\bar u))\right)
\\ \nonumber
&\stackrel{\eqref{ggg}_2}=& -\int_{\partial\Omega} \chi_{\{\bar u=0\}} [\bar \z,\nu] \d\mathcal H^{n-1} +\int_\Omega \chi_{\{\bar u=0\}} (\bar u-\bar f)
\\ \label{dueinuno} &\stackrel{\eqref{ggf}} =& -\int_\Omega \chi_{\{\bar u=0\}} \bar f.
\end{eqnarray}
Since $\bar f\ge 0$, the chain of inequalities in \eqref{dueinuno} implies that
$$
\bar f=0 \quad\mbox{ a.e. on $\{\bar u=0\}$}\quad\mbox{and}\quad
0=\lim_{b\to 0}\frac{1}{b}\left(\lim_{a\to 0}\int_\Omega \d (\bar\z, DT(\bar u))\right).
$$
Analogously, we of course obtain that $f=0$ a.e. on $\{u=0\}$. Therefore \eqref{ic} may be rewritten as
$$
\int_{\Omega} (u-\overline{u}) {\rm sign }_0^+(u- \bar u) \d x
\le \int_{\Omega} (f-\overline f) {\rm sign }_0^+(u- \bar u) \d x \le \int_{\Omega} (f-\overline f)^+ \d x
$$
and the proof is complete.
\end{proof}

\begin{remark}
A supersolution $\bar u$ of \eqref{pde} for $m>0$ may be defined as a function which satisfies all properties in Definition \ref{defi1d} besides \eqref{sm9d}, which is replaced by
$$
\bar u -f \ge \dive \bar \z \in L^\infty(\Omega),
$$
and \eqref{boundcondd=}, which is removed. With this definition, the proof of Theorem \ref{UniqEllipticd} continues to hold and yields $\bar u\ge u$. On the other hand, a subsolution $\underline u$ of \eqref{pde} may be defined as a function which satisfies all properties in Definition \ref{defi1d} besides \eqref{sm9d}, which is replaced by
$$
\underline u -f \le \dive \underline \z \in L^\infty(\Omega),
$$
and $\underline u \ge g$, which is removed. With this definition, the proof of Theorem \ref{UniqEllipticd} (with $u$ replaced by $\underline u$ and $\bar u$ replaced by $u$) continues to hold and yields $\underline u\le u$. Thus, as to the boundary conditions, supersolutions require only that $\bar u\ge g$ on $\partial\Omega$, whereas subsolutions require only that \eqref{boundcondd=} holds.

In the singular case $m<0$, analogous considerations lead to suitable definitions of sub and supersolutions for problem \eqref{pde}, for which the proof of the comparison principle stated in Theorem \ref{comp} continues to hold: in this case, supersolutions are only required to satisfy \eqref{boundcond=}, while subsolutions are only required to satisfy $\underline{u}\leq g$ on $\partial\Omega$.
\end{remark}

\section{Qualitative properties}\label{S-qp}

In this section we highlight some qualitative features of solutions to \eqref{pde}. Our interest is primarily concerned with their behavior as the boundary value $g$ becomes large.
As our analysis is based on comparison, we begin with a few examples of explicit solutions: in particular, constant solutions (which may not attain the boundary values) are given in (i) below; these coincide with solutions with large boundary values for $m<0$, whereas solutions with large boundary values for $m>0$ are given in (ii)-(iv).

\begin{lemma}\label{L-explicit}
Let $\Omega=B_R(0)$ for some $R>0$ and let $u$ be the  solution to \eqref{pde} with data $f=F\in [0,+\infty)$ and $g={G}\in {[}0,+\infty)$.
\begin{itemize}
\item[(i)]  If $m<0$, then $u={U}$ for all $G \geq U$, where $U{\in (0,+\infty)}$ is defined by $U-F = U^mN/R$. {If $m>0$, then $u=U$ for all $G\le U$, where $U\in [0,+\infty)$ is defined by $U+U^mN/R=F$.}

\smallskip

\item[(ii)] If $0<m<1$ and $G>F$ is sufficiently large, then
\begin{equation}
  \label{guess1stex01}u(x)=h(r)\chi_{B_r(0)}(\rho) +h(\rho)\chi_{\Omega\setminus B_r(0)}(\rho),\quad \rho:=\|x\|,
\end{equation}
where {$h\in C^1([r,R])$, positive and increasing, is the unique} solution to
\begin{equation}\label{B1101}
\left\{\begin{array}{l}
{m\dfrac{\d h}{\d \rho}=h^{1-m}(h-F)- (N-1)\dfrac{h}{\rho}}
\\[2ex]
h(R)=G
\end{array}\right.
\end{equation}
and $r\in (0,R)$ is the unique solution to
\begin{equation}\label{HN}
H_N(r):= h(r)-F-\frac{h^m(r)N}{r}=0.
\end{equation}

\item[(iii)] If $m=1$, $F=0$, and $R>N$, then
\begin{equation}\label{expl1}
u(x)=\left\{\begin{array}
  {ll} \displaystyle G \left(\frac{R}{N}\right)^{N-1} e^{N-R} & \quad {\rm if \ } \rho<N \\[2ex] \displaystyle G \left(\frac{R}{\rho}\right)^{N-1} e^{\rho-R} & \quad {\rm if \ } N<\rho<R.
  \end{array}\right.
\end{equation}

\item[(iv)] If $m>1$, $F=0$, and $G$ is sufficiently large, then $u=G$.

\item[(v)] If $m>1$, $N=1$, $F=0$, and $G<\left(\frac{m-1}{m}R\right)^{1/(m-1)}$ is sufficiently small, then
    $$
    u(x)=\left(G^{m-1}+\frac{1-m}{m}(R-\rho)\right)_+^\frac{1}{m-1}.
    $$

\end{itemize}

\end{lemma}

\begin{proof} Throughout the proof, primes denote differentiation with respect to the radial variable $\rho$. {Since all functions $u$ in $(i)$-$(v)$ are Lipschitz continuous, conditions \eqref{identify-w-new} and \eqref{identify-wd-new} are in fact equivalent to
\begin{equation}
\label{identify-w-6}
(\w,Du)=|Du|.
\end{equation}
}

(i). If $m<0$, let $u=U$ and $\w(x)=x/R$. Then $\z(x)= U^m x/R$ and $\dive \z= U^mN/R$, so that $u-F = \dive \z$ by the choice of $U$. Condition {\eqref{identify-w-6}} is obviously true. Finally, $[\z,\nu]= \frac{U^{m}x}{R}\cdot \frac{x}{R}= U^{m}$, hence  the boundary condition holds whenever $G\ge U$. The case $m>0$ is analogous, choosing $\w(x)=-x/R$.

\smallskip

(ii).   Recall here $0<m<1$; we look for a solution of the form \eqref{guess1stex01} with ${0< r< R}$ and $h\in C([r,R])$ nonnegative, nondecreasing and such that
$G=h(R)$.
We define $\w,\z\in X(\Omega)$ by
\begin{equation}\label{def-w-z}
\w(x):=\left\{\begin{array}{ll}\displaystyle\frac{x}{r} & {\rm if \ } \rho<r \\[2ex] \displaystyle\frac{x}{\rho} & {\rm if \ } \rho>r \end{array}\right.,\quad \z(x):=u^{m}\w=\left\{\begin{array}{ll}\displaystyle\frac{h^m(r) x}{r} & {\rm if \ } \rho<r \\[2ex] \displaystyle\frac{h^m(\rho)x}{\rho} & {\rm if \ } \rho>r. \end{array}\right.
\end{equation}
Then
$$
\dive\z=\left\{\begin{array}{ll}\displaystyle\frac{h^m(r) N}{r} & {\rm if \ } \rho<r \\[2ex] \displaystyle mh^{m-1}(\rho) {h'(\rho)} + h^m(\rho)\frac{N-1}{\rho} & {\rm if \ } \rho>r\,. \end{array}\right.
$$
Condition {\eqref{identify-w-6}} holds since
\begin{equation}\label{verid}
(\w,Du)=\left\{\begin{array}{ll} 0=|Du| & \mbox{if $\rho<r$} \\ {\frac{x}{\rho}\cdot h'(\rho)\frac{x}{\rho} = h'(\rho)} =|Du| & \mbox{if $\rho>r$}\,.\end{array}\right.
\end{equation}

The condition $h(R)=G$ in \eqref{B1101} implies that  $u=g$ on $\partial\Omega$, hence the boundary condition \eqref{boundcondd} holds.
The other conditions in \eqref{B1101} and  \eqref{HN}  implies that  $\dive \z ={u}-F$. It remains to check that  $h$ and $r$ exist and are unique. We discuss the cases $N=1$ and $N>1$ separately.

\smallskip

{\it Case $N=1$.} Since $G>F$, \eqref{B1101} has a unique solution $h$ in $(-\infty,R]$,  with $h$ increasing and $h(\rho)\to F$ as $\rho\to -\infty$ (observe that $h$ lies above the stationary solution $F$). Since $H_1(\rho)\to -\infty$ as $\rho\to 0^+$ and $H_1(R)>0$ for $G$ sufficiently large (recall that $m<1$), \eqref{HN} has a solution. Uniqueness of $r$ will be shown below for any $N\ge 1$.

\smallskip

{\it Case $N>1$.} We will argue that there exists a unique solution $h$ to \eqref{B1101} in $(0,R]$ with the following properties:
\begin{itemize}
\item[$(a)$] $h'(R)>0$;
\item[$(b)$] $h>F$ in $(0,R)$;
\item[$(c)$] $h$ has a unique minimum point $\rho_m\in (0,R)$.
\end{itemize}

$(a)$ follows immediately from \eqref{B1101} choosing $G$ sufficiently large (in particular, $G>F$). $(b)$ follows by contradiction: let $\rho_{0}$ be the closest point to $R$ at which $h(\rho)=F$; if $F>0$, by \eqref{B1101} we have $h'(\rho_{0})<0$ which, together with the fact that $h(R)=G>F$, contradicts the definition of $\rho_{0}$; if $F=0$ then $h$ is identically zero, in contradiction with the condition $h(R)=G$. In order to show $(c)$, assume by contradiction that $h'>0$ in $(0,R)$. Then we would have
$$
h^{1-m}\ge h^{-m}(h-F)>\frac{N-1}{\rho} \to +\infty \quad \mbox{as $\rho\to 0^+$},
$$
a contradiction. Therefore at least one point $\rho_{min}\in (0,R)$ exists with

\begin{equation}\label{gr1}
h'(\rho_{min})=0, \quad\mbox{that is,}\quad h(\rho_{min})-F=\frac{(N-1)h^m(\rho_{min})}{\rho_{min}}.
\end{equation}
Differentiating \eqref{B1101} and using \eqref{gr1}, one sees that $h''(\rho)>0$ at any point in which $h'(\rho)=0$. Therefore $\rho_{min}$ is unique and $h'(\rho)<0$  for $\rho\in (0,\rho_{min})$.
Since $H_N(\rho_{min})<0$ and $H_N(R)>0$ for $G$ sufficiently large (recall that $m<1$), there exists $r\in (\rho_{min},R)$ such that $H_N(r)=0$.

\smallskip

In order to show now that $r$ (the zero of $H_N$) is unique, we can reunify the cases $N=1$ and $N\ge 1$.  We have that
$$
H_N'(\rho)=h'(\rho)-\frac{N}{\rho}H_N(\rho)=H_N(\rho) \left(\frac{h^{1-m}}{m}-\frac{N}{\rho}\right) +\frac{h}{{m}\rho}\,.
$$
Then, since
$$
H_{N}(\rho)\leq h-\frac{Nh^{m}}{\rho}= mh^{m}\left(\frac{h^{1-m}}{m}-\frac{N}{m\rho}\right)<mh^{m} \left(\frac{h^{1-m}}{m}-\frac{N}{\rho}\right),
$$
it holds that
$$
H_N'(\rho)\geq {\frac{h}{m\rho}}>0 \quad {\rm if \ }H_N(\rho){\ge} 0.
$$
Therefore, there exists a unique $r\in (0,R)$ such that $H_N(r)=0$.

\smallskip

(iii). As in (ii), we look for a solution of the form \eqref{guess1stex01} with ${0 <r< R}$ and $h\in C([r,R])$ nonnegative, nondecreasing and such that $G=h(R)$.
We define $\w,\z\in X(\Omega)$ as in \eqref{def-w-z} and, as in (ii), we obtain that $h(r)=\frac{h(r) N}{r}$, i.e. $r=N$, and that $h$ satisfies
\begin{equation}\label{B11}
h= h' +(N-1)\frac{h}{\rho}\,,\quad N<\rho<R.
\end{equation}
The solution to \eqref{B11} can be computed explicitly, leading to \eqref{expl1}. Condition {\eqref{identify-w-6}} holds ({cf. \eqref{verid} and} note that $h$ is nondecreasing) and $u=G$ on $\partial\Omega$, hence \eqref{boundcondd} holds.

\smallskip

(iv). Let $\w(x)=x\frac{G^{1-m}}{N}$ with $G^{m-1}\ge \frac{R}{N}$, so that $\|\w\|_\infty\le 1$. Then $\z(x)= u^m(x) \w(x) = Gx$, so that $\dive \z =G=u$. Condition {\eqref{identify-w-6}} is obviously true and the boundary datum is attained.

\smallskip

(v). It suffices to define $r= R-\frac{m}{m-1}G^{m-1}$ and $\w,\z\in X(\Omega)$ by
\begin{equation*}
\w(x):=\left\{\begin{array}{ll}\displaystyle\frac{x}{r} & {\rm if \ } \rho<r \\[2ex] \displaystyle {\frac{x}{\rho}} & {\rm if \ } \rho>r \end{array}\right.,\quad \z(x):=u^{m}\w=\left\{\begin{array}{ll}\displaystyle 0 & {\rm if \ } \rho<r \\[2ex] \displaystyle h^m(\rho){\frac{x}{\rho}} & {\rm if \ } \rho>r \end{array}\right.
\end{equation*}
and to argue as in item (ii) {(with $F=0$ and $N=1$)}.
\end{proof}

We now draw a few consequences based on comparison. In the (scaling-wise) linear and super-linear case, $m\ge 1$, solutions blow-up uniformly in the whole domain as the boundary datum becomes large. In particular, no nontrivial large solution {can exist}.

\begin{proposition}
Let $m\ge 1$ and let $\Omega=B_R(0)$ for some $R>0$, with $R>N$ if $m=1$. Let $ f\in L^\infty(\Omega)$ nonnegative and $g_G\in L^\infty(\partial\Omega)$ such that $g_G\ge G$. Then the  solutions $u_G$ of \eqref{pde} with data $f$ and $g_G$ are such that
$$
u_G(x)\to + \infty \quad\mbox{for all $x\in B_R(0)$ as $G\to +\infty$.}
$$
\end{proposition}

\begin{proof}
In view of the comparison tool given by Theorem \ref{UniqEllipticd}, it suffices to prove the statement for $f=0$ and $g_G=G$. In this case solutions are explicitly given by Lemma \ref{L-explicit}(iii)-(iv), whence the result.
\end{proof}

On the contrary, in the (scaling-wise) singular case, $m<0$, solutions are bounded independently of their boundary value:

\begin{proposition}\label{L-q<0}
Let $m<0$ and let $\Omega=B_R(0)$ for some $R>0$. Let $f\in L^\infty(\Omega)$ and let $U$ be defined by $U-\|f\|_{L^{\infty}(\Omega)} = U^mN/R$. Then
$$
\|u\|_{L^{\infty}(\Omega)} \le U
$$
 for any nonnegative $g\in L^\infty(\partial\Omega)$, where $u$ is the  solution $u$ of \eqref{pde} with data $f$, $g$.
\end{proposition}

\begin{proof}
Let $u_U=U$ be the solution to \eqref{pde} with data $f=\|f\|_{L^{\infty}(\Omega)}$ and $g=G$ for all $G\ge U$, as given by Lemma \ref{L-explicit}(i). Then the conclusion follows choosing $G\ge \max\{U,\|g\|_{L^{\infty}({\partial\Omega})}\}$ and applying the comparison tool given by Theorem {\ref{comp}}.
\end{proof}

The (scaling-wise) sublinear case, $0<m<1$, lies somewhat in between, in the sense that solutions are {\it locally} bounded independently of the boundary value $g$.

\begin{proposition}\label{L-q01}
Let $0<m<1$, $R>0$, $\Omega=B_R(0)$, $f\in L^\infty(\Omega)$. Then
$$
u (x)\le \overline u (\|x\|),
$$
for any nonnegative $g\in L^\infty(\partial\Omega)$, where $u$ is the  solution of \eqref{pde} with data $f$, $g$ and
\begin{equation}\label{def-acca}
\overline u(x):=\overline h(r)\chi_{B_{\overline r}(0)}(\rho) +\overline h(\rho)\chi_{\Omega\setminus B_{\overline r}(0)}(\rho),\quad \rho:=||x||,
\end{equation}
where $\overline h = \overline v^\frac{1}{m-1}$, $\overline v$ is the unique solution to
\begin{equation}\label{def-v}
 {\dfrac{\d \overline v}{\d \rho}}=\frac{m-1}{m}\left(1-\|f\|_{L^{\infty}(\Omega)} \overline v^{1/(1-m)}-\frac{(N-1)\overline v}{\rho}\right), \quad \overline v(R)=0
\end{equation}
and $\overline r$ is the unique solution to $\overline h(r)-\|f\|_{L^{\infty}(\Omega)}=\overline h^m(r)N/r$.
\end{proposition}

\begin{proof}
Let $F=\|f\|_{L^{\infty}(\Omega)}$ and $G\ge \|g\|_{L^{\infty}({\partial\Omega})}$ sufficiently large. We consider the solutions $u$ with data $F,G$  obtained in Lemma \ref{L-explicit}(ii) and index solutions accordingly, i.e. we let $u=u_G$, $r=r_G$, and $h=h_G$.
{
Letting $v_G=h_G^{m-1}$, we see that $v_G$ solves
$$
{\dfrac{\d v_G}{\d \rho}}=\frac{m-1}{m}\left(1-Fv_G^{1/(1-m)}-\frac{(N-1)v}{\rho}\right), \quad v_G(R)=G^{m-1}\stackrel{G\to +\infty}\to 0.
$$
Hence, by standard ode theory, $v_G\to \overline v$ locally uniformly in $(0,R]$ as $G\to +\infty$ (in fact, uniformly in $[0,R]$ if $N=1$) with $\overline v$ as in \eqref{def-v}, and that $r_G$ converges to $\overline r$. Finally, it follows from Theorem \ref{UniqEllipticd} that $u_G\ge u$ for all $G$ sufficiently large, hence the result.
}
\end{proof}

Observe that, in Proposition \ref{L-q01}, one has that $\overline v\sim (1-m)(R-\rho)/m$ as $\|x\|\to R$; therefore,
$$
u\lesssim \left(\frac{1-m}{m}(R-\|x\|)\right)^\frac{1}{m-1} \quad\mbox{as $\|x\|\to R$}.
$$ This asymptotic upper bound   is optimal, as shown by the following proposition.

\begin{proposition}
Let $0<m<1$, $R>0$, $\Omega=B_R(0)$, $G>0$, $ f\in L^\infty(\Omega)$, and $g=g_G\in L^\infty(\partial \Omega)$ such that $g_G\ge G$. Then the corresponding  solutions $u_G$ of \eqref{pde} are such that
$$
\liminf_{G\to \infty} u_G(x) \ge \overline u_0(\|x\|) \quad\mbox{for all $x\in B_R(0)$},
$$
 where $\overline u_0$ is defined as in \eqref{def-acca} with $F=0$ in \eqref{def-v}, and is such that $\overline u_0\sim \left(\frac{1-m}{m}(R-\|x\|)\right)^\frac{1}{m-1}$ as $\|x\|\to R$.
\end{proposition}

\begin{proof}
In view of the comparison tool given by Theorem \eqref{UniqEllipticd}, it suffices to prove the statement for the explicit solutions obtained in Lemma \ref{L-explicit}(ii) with $f=0$ and $g=G$. The proof is identical to the one of the previous Lemma.
\end{proof}

Finally, we give two explicit examples of the {\it regularizing effect} given by Lemma \ref{cont}: solutions do not jump in the bulk, even if $f$ does.

\begin{example}\label{ex.cont.1}
Let $R>0$, $\Omega=B_R(0)$, $0<r<R$, $f=\alpha\chi_{B_{r}(0)}+\beta\chi_{B_R(0)\setminus B_{r}(0)}>0$ ($0<\beta<\alpha)$, and $g=\beta$. Then the  solution of \eqref{pde} is $u=\beta$ for all $r$ sufficiently small.

\smallskip

Let again $\rho:=\|x\|$. We choose
$$
        \w(x):=\left\{\begin{array}
          {ll} C x \, & {\rm if \ } \rho\leq r \\ C \frac{r^N x}{\rho^N} & {\rm if \ } r<\rho<R,
        \end{array}\right.
$$
so that $\w$ (hence $\z$) is continuous across $\rho=r$ and
$$
\dive \z = \left\{\begin{array}
          {ll} \beta^m C N \, & {\rm if \ } \rho\leq r \\ 0 & {\rm if \ } r<\rho<R,
        \end{array}\right.
$$
hence $u-f=\dive \z$ holds choosing $C=\beta^{-m}(\beta-\alpha)/N$. Finally, imposing $\|\w\|_{\infty}\le 1$ we obtain
$$
\frac{\alpha-\beta}{N\beta^m} r\le 1 \quad\mbox{and}\quad \frac{\alpha-\beta}{N\beta^m} \frac{r^N}{R^{N-1}}  \le 1,
$$
which are satisfied for all $r$ sufficiently small.
\end{example}

Combining this construction with the one in the previous results --through Bernoulli-type equations-- one could in fact provide explicit solutions for any $r\in (0,R)$  and any constant boundary value. We give a prototypical example in the special case $m=1$, $N=1$, where the solution is still explicit.

\begin{example}
Let $N=1$, $m=1$, $\Omega=B_R(0)$, $0<r<R$, $f=\alpha\chi_{B_{r}(0)}+\beta\chi_{\Omega\setminus B_{r}(0)}>0$ ($\alpha>\beta>0)$, and $g=G>0$. Arguing as in Example \ref{ex.cont.1}, we see that $u=\beta$ is the solution to \eqref{pde} {if} $G\leq \beta$ {and} $\frac{\alpha-\beta}{\beta}r\leq 1$. Instead, if $\frac{\alpha-\beta}{\beta}r> 1$,  we look for solutions of the form
$$
u(x)=A\chi_{B_{r}(0)}+h(\rho)\chi_{\Omega\setminus B_{r}(0)}, \quad \rho:=\|x\|,
$$
for a suitable $A>0$. We choose
$$
        \w(x):=\left\{\begin{array}
          {ll}  -\frac{x}{r} \, & {\rm if \ } \rho\leq r \\ - \frac{ x}{\rho} & {\rm if \ } r<\rho<R\,.
        \end{array}\right.
$$
By imposing to $u$ to solve problem \eqref{pde} we obtain $A=\alpha(\frac{r}{r+1})$ and
$$
- h'(\rho)=h(\rho) - \beta \quad \mbox{for $r<\rho<R$.}
$$
Integrating and imposing $h(r)=A$, we obtain
$$
h(\rho)= \beta +\left(\frac{\alpha r}{r+1}-\beta\right)e^{r-\rho}.
$$
Observe that the boundary condition is satisfied in the sense of Definition \ref{defi1d}  as soon as {$G\le \beta$
}
since $[w,\nu]=-1$ at $\rho=R$.
\end{example}

\section{Homogeneous Neumann boundary conditions and more general nonlinearities}
\label{neumann}

Existence and uniqueness results analogous to Theorems \ref{exi} and \ref{exid} hold for {\eqref{pde}} with homogenous Neumann boundary conditions. The definition of solution is the following one:

\begin{definition}\label{def-neumann}
Let $f\in L^\infty(\Omega)$ be nonnegative with $\inf f>0$ if $m<0$. A function $u:\Omega\to [0,+\infty)$ is a solution of problem ({\ref{neumannpde}}) with datum $f$ if $u\in {TBV^{+}(\Omega)}\cap L^\infty(\Omega)$ {and} there {exists} $\w\in L^\infty(\Omega;\rn)$ {such that:} $\|\w\|_{\infty}\leq 1${, }$\z:=u^m\w\in X(\Omega)$ satisfies
{
\begin{equation}\label{identify-wdn-new}
|D \phi({T_a^b}(u))| \le (\z,D T_a^b(u)) \quad \mbox{as measures for {a.e.}  $0<a<b\leq +\infty$},
\end{equation}
}
\begin{equation}\label{sm9dn}
u-f=\dive \z \quad\mbox{in $\mathcal D'(\Omega)$}\,,
\end{equation}
and
\begin{equation}\label{boundconddn}  [\z,\nu] = 0
 \quad\mbox{$\mathcal H^{N-1}$-a.e. on $\partial\Omega$.}
\end{equation}
\end{definition}

\begin{theorem}\label{exineumann}
Let $f\in L^\infty(\Omega)$ be nonnegative with $\inf f>0$ if $m<0$. Then there exists a unique solution $u$ of {\eqref{neumannpde}} with datum $f$ in the sense of Definition \ref{def-neumann}. {In addition, $u\in DTBV^+(\Omega)$,
\begin{subequations}\label{em-totn}
\begin{equation}\label{em1n}
{(\w,DT_a^b(u))=|D T_a^b(u)|} \quad\mbox{for a.e. $0<a<b\leq +\infty$,}
\end{equation}
and
\begin{equation}\label{em2n}
(\z,D T_a^b(u))=\left|D\phi(T_a^b(u))\right| \quad\mbox{for a.e. $0<a<b\leq +\infty$.}
\end{equation}
\end{subequations}
}
\end{theorem}

\begin{proof}[Sketch of the proof] The proof of Theorem \ref{exineumann} closely follows the lines of that of Theorems \ref{exi}, \ref{comp},  \ref{exid}, and \ref{UniqEllipticd}, with many simplifications due to the homogeneous Neumann boundary conditions. We only mention that one has to use the following approximating problems:
$$
\left\{
\begin{array}{ll}
u_\vare - f=\dive\left( (\vare + |u_\vare|)^m\frac{\nabla u_\vare}{|\nabla u_\vare|_\eps}+\eps \nabla u_\vare\right) & \mbox{in }\ \Omega
\\ \left((\vare + |u_\vare|)^m\frac{\nabla u_\vare}{|\nabla u_\vare|_\eps}+\eps \nabla u_\vare\right)\cdot \nu=0
 & \mbox{on }\ \partial\Omega,
\end{array}
\right.
$$
whose solutions $u_\eps$ satisfy
$$
\inf f \le u_{\vare} \leq \|f\|_{L^{\infty}(\Omega)}.
$$
The estimates and the passage to the limit in $\Omega$ are completely analogous, in fact simpler, due to the absence of boundary terms: for instance, in the proof of Lemma \ref{l-res-weak} one has to use lower semi-continuity of the functional
$$
u\in L^1(\Omega)\mapsto \left\{\begin{array}{cc}\displaystyle \int_\Omega \varphi d |D\phi_F(u)| & {\rm if \ } u\in BV(\Omega) \\ +\infty & {\rm otherwise} \end{array}\right., {\rm \ with \ }0\leq \varphi\in \mathcal D(\Omega)\,,
$$
(see \cite[Theorem 3.1]{AdCF_esaim07}) which does not contain any boundary contribution. The boundary condition \eqref{boundconddn} can be shown to hold as follows. The fluxes
$$
\z_\eps:= (\vare + |u_\vare|)^m\frac{\nabla u_\vare}{|\nabla u_\vare|_\eps}+\eps \nabla u_\vare
$$
satisfy (in view of \eqref{Green} and since $[\z_\eps,\nu]=0$ on $\partial\Omega$)
\begin{equation}\label{kl1}
0= \int_\Omega \varphi \dive \z_\eps -\int_\Omega \z_\eps\cdot \nabla \varphi \quad\mbox{for all $\varphi\in C^\infty(\overline\Omega)$}
\end{equation}
and are such that $\z_\eps\rightharpoonup \z$ in $L^2(\Omega;\R^N)$ and $\dive\z_\eps \stackrel{*}\rightharpoonup \dive\z$ in $\mathcal M(\Omega)$. Hence, passing to the limit as $\eps\to 0$ in \eqref{kl1} we obtain
$$
0= \int_\Omega \varphi \dive \z -\int_\Omega \z\cdot \nabla \varphi \stackrel{\eqref{Green}}= \int_{\partial\Omega} \varphi [\z,\nu] \d\mathcal H^{N-1},
$$
for all $\varphi\in C^\infty(\overline\Omega)$, implying that $[\z,\nu]=0$ on $\partial\Omega$.
\end{proof}

\begin{remark}\label{rem-cont}
The arguments in Lemmas \ref{cont} and \ref{4.8}, leading to a null singular set, apply also to the resolvent equation of other parabolic equations with linear growth lagrangian, equations, such that of the relativistic heat equation ($m=1$) and the relativistic porous medium equation ($m>1$, cf. \eqref{pmrhe}),
  \begin{equation*}
  u-f=\dive\left(\frac{|u|^m \nabla u}{\sqrt{u^2+|\nabla u|^2}}\right)\quad {\rm for \ } m{\ge} 1,
  \end{equation*}
  or that of the speed-limited porous medium equation (cf. \eqref{flpme}),
  \begin{equation*}
  u-f=\dive\left(\frac{|u|\nabla  u^{M-1}}{\sqrt{1+|\nabla u^{M-1}|^2}}\right)\quad {\rm for \ } M>{1},
  \end{equation*}
  studied in \cite{ACM_na05,ACMM_ma10,Caselles_jde11,CC_na13} under different types of boundary conditions (compare condition \eqref{identify-zd} with (3.26) in \cite{ACM_na05},  (34)  in \cite{ACMM_ma10}, {(50)} in \cite{Caselles_jde11}, and {condition 3 of Definition 8.3} in \cite{CC_na13}). Therefore, the unique solutions of those problems belong as well to $DTBV^+(\Omega)$. {Note, however, that the proof of Lemmas \ref{cont} and \ref{4.8} does {\em not} carry over to $m=0$, where indeed solutions may have jumps.}
\end{remark}

\begin{remark}
Throughout the paper, we have focused on  the case of a  mobility given by the nonlinear term $u^{m}$. However, the proofs of both existence and uniqueness of solutions for both problem \eqref{pde} and problem \eqref{neumannpde} still hold in the case of a more general nonlinearity:
$$
\left\{
\begin{array}{ll}
u-f=\dive\left(\phi'(u)\frac{\nabla u}{|\nabla u|}\right) & \mbox{in }\ \Omega\\
u=g & \mbox{on}\  \partial\Omega\,
\end{array}
\right.
$$
(we use $\phi'$ for consistency with \eqref{def-phi-I}),
where either
\begin{itemize}
\item[(S)] $\phi'(s)$ is a locally continuous strictly decreasing function on $(0,+\infty)$
\end{itemize}
or
\begin{itemize}
\item[(D)] $\phi'(s)\in {C}([0,+\infty))$ is a strictly
{increasing
}
function.
 \end{itemize}
Of course, (S) and (D) represent the singular case ($m<0$) and the degenerate case ($m>0$) of the previous sections, respectively. The respective assumptions on $f$ and $g$ are identical (for instance, in case (S) one asks that $g$ be strictly positive on $\partial\Omega$). Definitions \ref{defi1}, respectively \ref{defi1d}, can be modified accordingly, by formally substituting $u^{m}$ with $\phi'(u)$.
\end{remark}

\bigskip
\footnotesize
\noindent\textit{Acknowledgments.}
 The second and third author acknowledge partial support by the Spanish MEC and FEDER project MTM2015-70227-P. The third author has been partially supported by the Gruppo Nazionale per l'Analisi Matematica, la Probabilit\`a e le loro Applicazioni (GNAMPA) of the Istituto Nazionale di Alta Matematica (INdAM).

\end{document}